\documentclass{amsart}

\usepackage{amsfonts,amsmath,amssymb,amscd,amsthm}
\usepackage{mathrsfs}
\usepackage{faktor}    
\usepackage[utf8]{inputenc}
\usepackage{csquotes}
\usepackage{microtype}
\usepackage[english]{babel}
\usepackage{bbm}
\usepackage{pst-node}
\usepackage{enumitem}
\usepackage{tikz-cd}
\usepackage{extpfeil}
\usepackage{tikz}
\usepackage{comment}
\usetikzlibrary{arrows, automata, positioning}
\usepackage{tabto}
\TabPositions{2 cm, 4 cm, 6 cm, 8 cm}
\usepackage[unicode, pdftex]{hyperref}
\hypersetup{
  colorlinks   = true, %Colours links instead of ugly boxes
  urlcolor     = blue, %Colour for external hyperlinks
  linkcolor    = blue, %Colour of internal links
  citecolor   = red %Colour of citations
}

\DeclareMathOperator{\tr}{tr}

\DeclareMathOperator{\supp}{supp}

\newcommand{\Altfin}[1]{\mathrm{Alt}_{\mathrm{fin}}\left(#1\right)}
\newcommand{\SLfin}[1]{\mathrm{SL}_{\mathrm{fin}}\left(#1\right)}
\newcommand{\Symfin}[1]{\mathrm{Sym}_{\mathrm{fin}}\left(#1\right)}

\newcommand{\Ch}[1]{\mathrm{Ch}\left(#1\right)}
\newcommand{\Tr}[1]{\mathrm{Tr}\left(#1\right)}

\DeclareMathOperator{\fin}{\mathrm{fin}}

\DeclareMathOperator{\AltZ}{Alt_{\fin}(\Z)}
\DeclareMathOperator{\SymZ}{Sym_{\fin}(\Z)}
\DeclareMathOperator{\Alt}{Alt}
\DeclareMathOperator{\Sym}{Sym}

\DeclareMathOperator{\SL}{SL}

\newcommand{\RF}{{\texttt{RF}}}
\newcommand{\LEF}{{\texttt{LEF}}}
\newcommand{\MAP}{{\texttt{MAP}}}
\newcommand{\SR}{{\texttt{SR}}}

\newcommand{\F}{{\mathbb{F}}}

\newcommand{\N}{{\mathbb{N}}}
\newcommand{\Z}{{\mathbb{Z}}}
\newcommand{\C}{{\mathbb{C}}}

\newtheorem{theorem}{Theorem}[section]
\newtheorem{corollary}[theorem]{Corollary}
\newtheorem{lemma}[theorem]{Lemma}

\newtheorem{proposition}[theorem]{Proposition}

\newtheorem*{lemma*}{Lemma}
\newtheorem*{proposition*}{Proposition}
\newtheorem*{theorem*}{Theorem}
\newtheorem*{corollary*}{Corollary}
\newtheorem*{claim*}{Claim}

\theoremstyle{definition}
\newtheorem{definition}[theorem]{Definition}
\newtheorem{example}[theorem]{Example}
\newtheorem*{definition*}{Definition}

\theoremstyle{remark}
\newtheorem{remark}[theorem]{Remark}

\theoremstyle{definition}
\newtheorem{question}[theorem]{Question}

\newcounter{partnumber}
\newcommand{\mypart}[1]{
  \refstepcounter{partnumber}
    \bigskip
   { \center \textbf{\large Part \Alph{partnumber}: #1}} 
    \bigskip
    \addcontentsline{toc}{part}{Part \Alph{partnumber}: #1} %
}

\title{Characters of diagonal products and Hilbert--Schmidt stability}
\author{Alon Dogon, Arie Levit and Itamar Vigdorovich}
\begin{document}

\begin{abstract}
We  initiate a quantitative study of Hilbert--Schmidt stability for infinitely presented groups through the novel notion of stability radius growth.
We exhibit an uncountable family of Hilbert--Schmidt stable amenable groups with  arbitrarily large such growth. In particular, this answers  a question of Lubotzky. Our approach is based on the character-theoretic stability criterion of  Hadwin and Shulman.
We classify the characters  of alternating and elementary enrichments as well as diagonal products, including the classical family of  B.H. Neumann groups. 
%We also derive local Hilbert--Schmidt stability of those  enrichments.
\end{abstract}

\maketitle
\section{Introduction}

The study of stability  for groups originated with Ulam \cite{Ulam}, who asked the following general question: is a map between groups which  is close to being a homomorphism, necessarily  close to a \emph{true} homomorphism? We focus on the notion of Hilbert--Schmidt stability, specializing this question to unitary representations. We will be specifically  interested in certain  quantitative aspects of this theory.
%of stability by introducing a notion of stabili.%  contribute to the quantitative study of Hilbert--Schmidt stability, by introducing the notion of \emph{stability radius growth}.

%In this work , as well as establish Hilbert--Schmidt stability for uncountable families of groups.
%Roughly speaking, we are interested in the following question: Given a group $G$ and a map $\varphi : G \to U(d)$ which is almost a unitary representation, how close is $\varphi$ to a true unitary representation?
%We introduce which measures the length of group relations that need to be tested on $\varphi$ before there is any hope of finding an honest representation up to the desired degree of precision.
%The growth type of the stability radius growth function turns out to be a group invariant, and we prove the existence of uncountably many $2$-generated Hilbert--Schmidt stable groups with different such growths (Theorem \ref{intro:thm:arbitrary large SRad}).
%Let us mention that quantitative Hilbert--Schmidt stability has recently gained interest, as it proved to be a key ingredient in the resolution of the Connes embedding problem \cite{MIP*=RE}, \cite{delaSalle}, \cite{CVY}.

We now formally define what it means for a map to be close to a unitary representation. Consider the finite-dimensional unitary groups $U(d)$ with the \emph{normalized Hilbert--Schmidt distance}
$$ d_\mathrm{HS}(A,B) = \sqrt{\frac{1}{d} \Tr{(A-B)^*(A-B)}} \quad \forall A,B \in U(d).$$

Let $G$ be a group with a finite generating set $S$. For each $n\in \N$ denote by $B_S(n)$ the ball of radius $n$ at the identity element of $G$ with respect to the word metric.

\begin{definition}
\label{def:delta_homo}
Given some  $n \in \mathbb{N}$ and $\delta > 0$, an \emph{$(n,\delta)$-almost representation} of the group $G$ is a map $\varphi : B_S(n) \to U(d)$ for some $d \in \mathbb{N}$ satisfying
$$d_\mathrm{HS}(\varphi(gh), \varphi(g)\varphi(h)) \leq \delta$$
for all pairs of elements $g,h \in B_S(n)$ with $gh \in B_S(n)$.
\end{definition}

Hilbert--Schmidt stability \cite{HS_grp} means that almost representations necessary arise as small deformations of  honest representations, in the following precise sense.

\begin{definition}
\label{def:intro HS stability}
 The group $G$ is called \emph{Hilbert--Schmidt stable} if for every $\varepsilon > 0$ there exist $\delta > 0$ and $n \in \N$ such that the following holds: for every $d \in \N$ and every $(n,\delta)$-almost representation $\varphi : B_S(n) \to U(d)$ there exists an honest unitary representation $\pi : G \to U(d)$ satisfying
$$  d_\mathrm{HS}(\varphi(s),\pi(s)) < \varepsilon \quad \forall s \in S. $$
\end{definition}

The notion of Hilbert--Schmidt stability is independent of the choice of the particular  generating set $S$, see \cite{HS_grp,BLT,DogonAlon2021SaAo} for more background.

We turn our attention  to quantifying Hilbert--Schmidt stability. One natural way to do so is to quantify $\delta$ in terms of $\varepsilon$.
This   approach leads to the notion of \emph{stability rate}   \cite{Glebsky, becker2021abelian}, which was defined only for finitely presented groups. For infinitely presented groups, one needs to take into account the fact that the number  of different group relations that need to be tested might grow as $\varepsilon$ becomes smaller. 
We capture this behaviour by means of the following novel notion.

\begin{definition}
    \label{def:SRad}
    The \emph{Hilbert--Schmidt stability radius growth} of the group $G$ with respect to its finite generating set $S$ is the function 
    $$\SR^S_G: (0, \infty) \to \mathbb N \cup \{ \infty \}$$ where for each $M>0$ we let $\SR_G^S(M)$ be the smallest integer $n \in \mathbb{N}$ such that there exists a     $\delta > 0$ for which the pair $(n,\delta)$  satisfies the conclusion of Definition \ref{def:intro HS stability}  with respect to the parameter $\varepsilon = 1/M$. If no such $n$ exists, set $\SR_G^S(M)= \infty$.
\end{definition}

The asymptotic growth type of the function
$\SR_G^S$ does not depend on the particular generating set $S$. As such, this growth type is a group invariant, to be denoted $\SR_G$.
It follows from the definitions that the  finitely generated group $G$ is Hilbert--Schmidt stable if and only if $\SR_G(M) < \infty$ for  all $M > 0$. 

It is important to note that the stability radius growth $\SR_G$ is an interesting invariant only in the infinitely presented case. In fact, if the group $G$ is finitely presented and Hilbert--Schmidt stable   then  $\SR_G$ is bounded. For all this see \S\ref{sec:stab rad}.
%InThe More precisely  $\SR_G$ measures a number of group relations to be certified, in a sense wThis resonates with the above-mentioned intuiotintuition can be made precise (see e.g. \cite[Lemma 3.12]{BLT}).

Our first main result shows that the Hilbert--Schmidt stability radius of infinitely presented amenable groups can grow arbitrarily fast (see our standing notations on page \pageref{subsec:notation} for the asymptotic comparison symbol $\preceq$).

\begin{theorem} \label{intro:thm:arbitrary large SRad}
Let $f: \N \to \N$ be an arbitrary monotone non-decreasing growth function. Then there exists an $2$-generated Hilbert--Schmidt stable amenable group $G_f$ with $f \preceq \SR_{G_f}$. 
\end{theorem}

As an immediate corollary of Theorem \ref{intro:thm:arbitrary large SRad}, we obtain the existence of uncountably many pairwise non-isomorphic $2$-generated Hilbert--Schmidt stable groups. This answers a question of Alex Lubotzky, which was raised during the workshop "Geometric Aspects of Flexible Stability"  at the Erdős Center in 2022. %The  permutation stability analgoue is \cite{LL2}, which can be seen as a (partial) precursor to this work.

\subsection*{Diagonal products}

A large supply of  infinitely presented Hilbert--Schmidt stable groups can be provided by using \emph{diagonal products}  \cite{KP,BZ}.
This is a flexible construction which takes as input a given \emph{LEF (locally embeddable into finite) group} \cite{VG}. A group $G$ is called LEF if the multiplication table of every finite subset of it can be represented in some finite group.

More precisely, the construction of a diagonal product depends not only on the given LEF group $G$, but also on a choice of a generating set $S$ as well as a particular LEF approximation, i.e. a sequence of maps $\theta_n : G \to G_n$ into some finite groups $G_n$ which are eventually multiplicative and injective (see Definition \ref{def:partial homo and LEF approx}). The diagonal product associated to this data is the subgroup of the direct product $\prod G_n$ generated by the elements $(\theta_n(s))$ for each $s \in S$ (see Definition \ref{def:diagonal product}). It is denoted\footnote{We have preferred to keep the (crucial) dependence of the diagonal product $\bigotimes G_n$ on the generating set and on the LEF approximation implicit in our notation.} $\bigotimes G_n$.

Our second main result shows that under a   mild technical hypothesis (i.e. that the LEF approximation is essential and by simple groups, see Definition \ref{def:essential LEF}), Hilbert--Schmidt stability passes to diagonal products.

\begin{theorem}\label{intro:thm:diag_prod_stab_main}
    Let $G$ be a LEF group with a fixed finite generating set $S$ and a LEF approximation $\theta_n : G \to G_n$. Assume that the LEF approximation $\theta_n$ is essential and by simple groups. If the group $G$ is amenable and Hilbert--Schmidt stable then so is the associated diagonal product $\bigotimes G_n$.
\end{theorem}

Once Theorem \ref{intro:thm:diag_prod_stab_main} on stability is established, Theorem \ref{intro:thm:arbitrary large SRad} follows by working with the lamplighter-based diagonal products introduced by Bradford \cite{Bradford_RF_growths}. Another ingredient towards Theorem \ref{intro:thm:arbitrary large SRad} is  a quantitative relation of stability radius growth to the invariants of LEF growth and MAP growth; see \S\ref{sec:stab rad} for further details.

In addition, using a slightly more sophisticated technique, we are able to show that the following classical family of diagonal products is Hilbert--Schmidt stable.

\begin{theorem}
\label{thm:intro:BH neumann and generalized are stable}
Let  $\mathcal{P} = \{d_n\}$ be any  strictly increasing  sequence of  integers.
\begin{enumerate}
    \item Assume that $d_n$ is odd for all $n \in \N$ and $d_1 \ge 5$. The classical B.H. Neumann group $ \bigotimes \Alt (d_n)$  is Hilbert--Schmidt stable; see  Example \ref{example:LEF approximation giving rise to BS}.
    \item Let $F$ be a finite field. Assume that $d_1 \ge 2$. The diagonal  product $\bigotimes \mathrm{SL}_{d_n}(F)$  is Hilbert--Schmidt stable; see Example  \ref{exam:LEF approximation elementary}.
\end{enumerate}
\end{theorem}

We remark that the  classical B.H. Neumann groups \cite{Neu} were shown to be stable in permutations in \cite{LL2}. This family is uncountable.

\subsection*{Character theory}
The key to Theorem \ref{intro:thm:diag_prod_stab_main} is a study of  \emph{characters} of diagonal products. Indeed, in the setting of amenable groups, Hilbert--Schmidt stability and character theory are related by a fundamental result of Hadwin and Shulman \cite[Theorem 4]{HS_grp}. The result says that Hilbert--Schmidt stability can be reformulated  as an approximation property for characters.  

The character theory of infinite groups was first studied by Thoma  \cite{Thoma-characters,Thoma-symmetric}. Here are the key notions of the theory.
A \emph{trace} on a countable group $G$ is a  conjugation-invariant and positive-definite function $\varphi : G \to \mathbb{C}$ normalized so that $\varphi(e_G) = 1$.  The  extreme points of the space  of all traces on the group $G$ are called \emph{characters}. 
The space  of characters of $G$  is  called the \emph{Thoma dual}  and is denoted $\Ch{G}$. 
Thoma singled out $\Ch{G}$ as an important object in the harmonic analysis on  $G$, for it captures the representation theory of $G$ into tracial von Neumann algebras.
A fundamental problem in the theory is to fully describe the space  $\Ch{G}$, i.e. \emph{character classification}. We refer to \S\ref{sec:characters} for more details.

The following result is a character classification for diagonal products.

\begin{theorem}
\label{thm:intro compatbile characters}
Let $G$ be a LEF group with a fixed finite generating set $S$ and a LEF approximation $\theta_n : G \to G_n$. Assume that the LEF approximation $\theta_n$ is essential and by simple groups. Denote by  $\Gamma = \bigotimes G_n$ the associated diagonal product. The space of characters of $\Gamma$ is
$$\Ch{\Gamma}   \cong \{ (\varphi, \psi) \in \prod_n \Ch{G_n} \times \Tr{G} \: : \: \text{$\psi$ is a compatible character for  $\varphi$} \}.$$
\end{theorem}

The  intricate notion of \emph{compatible characters}  is spelled out in detail in \S\ref{sec:diagonal products}. In any case,    the direct sum $\bigoplus G_n$ is a normal subgroup of the diagonal product $\bigotimes G_n$ in the setting of Theorem \ref{thm:intro compatbile characters}. Furthermore $\prod_n \Ch{G_n} \cong \Ch{\bigoplus_n G_n}$ by  Proposition \ref{prop:char_direct_prod}. Hence
 the theorem  shows that the space $\Ch{\Gamma}$ is \enquote{fibered} over the character space of the direct sum of finite groups $\bigoplus_n G_n$.
The fibers themselves can be  quite mysterious,   see  for instance the discussion in Example \ref{example:Young diagrams}.

There are two additional classes of groups we study in the context of character theory, namely the  so called \emph{alternating and elementary enrichments}  \cite{Bradford_LEF_ext}. 
These are close cousins of the lamplighter groups, falling  under the more general framework of halo products introduced in \cite{GT}. The  alternating enrichment of a group $G$ is the semidirect product $\mathscr{A}(G) = G \ltimes \Altfin{G}$, where $\Altfin{G}$ is the group of finitely supported even permutations on the underlying set of the group $G$, and the conjugation comes from the left action of $G$ on this set. The elementary enrichment is $\mathscr{E}(G) = G \ltimes \SLfin{V_G}$, where $\SLfin{V_G}$ is the group of unimodular finite rank linear transformations on the vector space $V_G$ over a finite field and with  basis indexed by the underlying set of the group $G$. See Definitions \ref{def:alternating enrichment} and \ref{def:elementary enrichment} for details.

%Consider the alternating enrichment $\mathscr A (G)$ or the elementary enrichment $\mathscr E(G)$ of the group $G$. This  is the semi-direct product  $\Gamma_0 = G \ltimes N$, where the normal subgroup is given by $N = \Altfin{G}$ in the alternating case and by $N = \mathrm{SL}(V_G)$ in the elementary case. Here $V_G$ is a vector space over some finite field with basis indexed by $G$. See \S\ref{sec:alternating} and \S\ref{sec:elementary} for more details on these two constructions.

Here is our character classification result for these enrichments. Recall that a group is called ICC if every non-trivial conjugacy class is infinite.

\begin{theorem}
\label{thm:intro:characters of enriched}
Let $\Gamma$ be either  the alternating enrichment $\mathscr A(G)$ or the elementary enrichment $\mathscr E(G)$ of a countably infinite group $G$. Let the normal subgroup $N \lhd \Gamma$ be either  $\Altfin{G}$ or $\SLfin{V_G}$ respectively. If the group $G$ is ICC then 
 \[
\Ch {\Gamma} \cong\left(\Ch G\sqcup\Ch{N}\right)/\left(\delta_{e}^{G}\sim 1_{N}\right).
\]
If the  group $G$ is not ICC then
\[
\Ch {\Gamma} \cong\Ch G\sqcup \left(\Ch{N}\backslash \{1_{N}\} \right).
\]
\end{theorem}

Here $\delta_e^G$ denotes the Dirac function at the identity on $G$, and $1_N$ denotes the trivial character of $N$. 
Note that the character space of both groups $\Altfin{\infty}$ and $\SLfin{\infty}$ is well-understood, see  \cite{Thoma-symmetric}  and \cite{skudlarek1976unzerlegbaren} respectively.

Theorem \ref{thm:intro:characters of enriched} allows us to prove  that  classical B.H. Neumann groups \cite{Neu} are Hilbert--Schmidt stable (Theorem \ref{thm:intro:BH neumann and generalized are stable}), even though they do not fall under the framework of Theorem \ref{intro:thm:diag_prod_stab_main}. 
This is done as part of a more general setting of \emph{alternating and elementary diagonal products}, which we introduce in \S\ref{sec:diagonal products from alternating and elementary enrichments}. See Theorem \ref{thm:alt diag prod is stable} for the full-fledged stability result in that context.

\subsection*{Local Hilbert--Schmidt stability}
\label{subsec:local HS stability}
An interesting local variant of Hilbert--Schmidt stability was introduced quite recently by Fournier-Facio, Gerasimova and Spaas \cite{FFGS}. It extends the  study of stability to potentially non-residually finite LEF groups.
Here is the relevant definition:

\begin{definition} \label{def:local stability}
A group $G$ with a finite generating set $S$ is \emph{locally Hilbert--Schmidt stable} if for every $\varepsilon > 0$ and $r \in \N$ there exist $\delta > 0$ and $n \in \N$ such that the following holds:
for every $d \in \N$ and every $(n,\delta)$-almost representation $\varphi : B_S(n) \to U(d)$ there exists a map $f : B_S(r) \to U(d)$ satisfying 
$$   d_{\mathrm{HS}}(\varphi(s), f(s)) < \varepsilon \quad \forall s \in S$$
as well as $f(gh) = f(g) f(h)$  for every pair of elements $g,h \in B_S(r)$ with $gh \in B_S(r)$.
\end{definition}

The notion of local Hilbert--Schmidt stability is of course independent of the choice of the finite generating set $S$.

\begin{theorem}
 \label{thm:intro:local stability passes to enrichments}
Let $G$ be a finitely generated amenable group. Then $G$ is locally Hilbert--Schmidt stable if and only if its  alternating enrichment $\mathscr A(G)$ is locally Hilbert--Schmidt stable. The same statement holds for the elementary enrichment $\mathscr E(G)$.
\end{theorem}

This  answers  the Hilbert--Schmidt version of a question of Bradford for amenable groups \cite[Question 6.3]{Bradford_stab}. It can be seen using  Theorems \ref{thm:intro:BH neumann and generalized are stable} and \ref{thm:intro:local stability passes to enrichments}  that alternating enrichments of classical B. H. Neumann groups  are locally Hilbert--Schmidt stable. Since alternating enrichments of infinite groups are never residually finite, we obtain an uncountable family of $4$-generated amenable Hilbert--Schmidt groups that are not Hilbert--Schmidt stable (see Corollary \ref{cor:uncountably many locally not HS stable}). This reproves the result \cite[Theorem E]{FFGS} using a different family of groups.

\subsection*{An application to characters of alternating groups}

As an aside to the main flow of ideas, we present another interesting and unexpected application of Theorem \ref{thm:intro compatbile characters}. It is an asymptotic result in the  classical character theory of alternating groups.

\begin{theorem}
\label{thm:cyclic structure}
Let $\sigma$ be an even cycle structure with $ |\mathrm{supp}(\sigma)| = q$. It represents  a single conjugacy class in the alternating group  $\Alt(n)$ for all $n \ge q + 2$. Let $n_i \in \mathbb{N}$ be a strictly increasing sequence of integers with $n_1 \ge q + 2$  and $\varphi_i \in \Ch{\Alt(n_i)}$ be arbitrary irreducible normalized characters. Then $ \lim_{n\to\infty} \prod_{i=1}^n \varphi_i(\sigma)$ exists.
\end{theorem}

 For instance,
 let $z_i\in \C$ be values from the normalized character tables of the alternating groups $\Alt(n_i)$  all corresponding to some fixed cycle structure. If $z_i \to z$ for some value $z \in \C$ with $|z| =1$ then we obtain $z = 1$. In other words, the complex number $1$ is the only accumulation point of such normalized character values on the unit circle.
 An analogous conclusion is  true also in the context of  the special linear groups $\mathrm{SL}_n(F)$ where $F$ is a finite field. However, in that case this already follows from existing character bounds (at least for $|F|$ sufficiently large, see \S\ref{sec:elementary}).

\subsection*{Related works}

In recent years, works dealing with Hilbert--Schmidt stability appeared in the fields of group theory, operator algebras and quantum information theory \cite{HS_grp, ISW, Ioana_HS, delaSalle, levit2023characters, ES}. This is in part due to the intimate relationship of these notions with approximation problems such as the Connes embedding problem, as well as the existence of non-sofic groups \cite{CVY, Dog}.
Hilbert--Schmidt stability is a special case of the more general setting of stability with respect to a family of metric groups \cite{ThomICM}.
There are many works on the parallel notion of permutation stability \cite{GR,AP, BLT, LL1, LL2}.
In a private communication, we have also learned from Henry Bradford that he has proven \emph{permutation stability} for some of the groups constructed in \cite{Bradford_RF_growths}.

Character classification for infinite discrete amenable groups has been studied since the foundational works of Thoma \cite{Thoma-symmetric, Thoma-characters}.  For a partial list of examples, see \cite{skudlarek1976unzerlegbaren,howe1977representations,vershik1981asymptotic,carey1984characters,levit2023characters}.

\subsection*{Standing notations}\label{subsec:notation} We set up some notations to be used throughout this work.
\label{notations}
\begin{itemize}
    \item We let $\left[n\right] = \{1,\ldots,n\}$ for every $n \in \mathbb{N}$.
    \item For $n \in \N$, we will denote by $\Alt(n) = \Alt([n])$ the alternating group on $[n]$.
    \item Let $g,h \in G$ be a pair of group elements. We use the notation $g^h = hgh^{-1}$ for the conjugate of the element $g$ by the element $h$.
    \item Let $f,g: (0,\infty) \to (0,\infty)$ be a pair of  monotone non-decreasing functions.
    \begin{itemize}
        \item We write $f \preceq g$ if there is a constant $C \ge 1$ such that  $f(x) \leq Cg(Cx)$ for all $x>0$. 
        \item We write $f \approx g$  if $f \preceq g$  and $g \preceq f$. In that case  $f$ and $g$ are said to be \emph{equivalent}.
        \item We use the same terminology   to compare functions defined only on natural numbers.
    \end{itemize}
\end{itemize}

\newpage
\subsection*{Acknowledgements}
This paper would not have been written without the towering inspiration and influence of Alex Lubotzky. It is only through his knowledge of the \enquote{classics} that we have become aware of the beautiful construction of B.H. Neumann. In addition, he asked the question which was the impetus for our collaboration.  

We would like to thank Henry Bradford, for bringing his construction of groups with prescribed large residual finiteness radius growth \cite{Bradford_RF_growths} to our attention as well as for  interesting conversations regarding quantitative stability. 
%His proof that some of these groups are permutation stable has inspired our work on Theorem 

We would like to thank Francesco Fournier-Facio, who brought the possible applications to local Hilbert--Schmidt stability to our attention.

We would like to thank Doron Puder for insightful conversations on the character theory of symmetric and alternating groups. He made several useful suggestions which were instrumental towards \S\ref{sec:alternating}. 

We would like to thank Noam Atar for bringing the relevant reference \cite{skudlarek1976unzerlegbaren} to our attention. 

We would like to thank Michael Chapman for helpful comments on an earlier draft of this work.

\vspace{20pt}

\subsection*{Funding information}
AD was supported by a Clore Scholars grant and ERC grant no. 772839 and no. 882751. AL was supported by ISF grant 1788/22 and BSF grant no. 2022105. IV was supported by ERC advanced grant no. 101052954 and by NSF postdoctoral fellowship grant DMS-2402368.

\vspace{\fill}

%\newpage

\enlargethispage{1\baselineskip}
\setcounter{tocdepth}{1}
\tableofcontents

\newpage
\mypart{Group theory --- diagonal products and enrichments}

This part is purely group-theoretical. We introduce the notions  of diagonal products, alternating and elementary enrichments as well as classical and generalized B.H Neumann groups. 

\section{Diagonal products}
\label{sec:LEF groups and diagonal products}

Let $G$ be a countable group.

\begin{definition}\label{def:partial homo and LEF approx}
A \emph{partial homomorphism} is a sequence of maps $\theta_n : G \to G_n$ into some  sequence of groups $G_n$ such that every pair of elements $g,h \in G$ satisfies $\theta_n(gh) = \theta_n(g) \theta_n(g)$ for all $n$ sufficiently large.

A \emph{LEF approximation} is a partial homomorphism such that the target groups $G_n$ are all \emph{finite} and such that 
 every non-trivial element $g\in G$ satisfies $\theta_n(g) \neq e$  for all $n$ sufficiently large. 
The group $G$ is called \emph{locally embeddable into finite groups (LEF)} if it admits a LEF approximation. 
\end{definition}

For example, every residually finite group is LEF.

Assume that the group $G$ is finitely generated and let $S$ be a fixed generating set.

\begin{definition}
\label{def:diagonal product}
Let  $\theta_n: G \to G_n$ be a LEF approximation of the group $G$.   The \emph{diagonal product} $\bigotimes_n G_n$ associated to  this data is the subgroup of the direct product $\prod_n G_n$ generated by the elements $\gamma_s = (\theta_n(s))_n$ for each $s \in S$. 
\end{definition}

Note that the diagonal product construction crucially depends on the fixed LEF approximation. This dependence is kept implicit in our notation, in the hope that should not cause any confusion. For more information on diagonal products we refer to \cite{KP,Mimura}.

In what follows, it will be very convenient to work with LEF approximations satisfying   following additional conditions.

\begin{definition}
\label{def:essential LEF}
A given LEF approximation $\theta_n : G \to G_n$ is called an \emph{essential LEF approximation by simple groups} if
\begin{itemize}  
    \item the finite groups $G_n$  are simple and pairwise non-isomorphic,
    %\footnote{More generally, it suffices to assume that the groups $G_n$ pairwise have no isomorphic composition factors in their composition series.}. \ad{Do we want this footnote? I think it is hard to guarantee this as we have not checked all the details pass through for disjoint Jordan-Holder content in our results}

    \item the restriction of $\theta_n$ to the generating set $S$ does not extend to a homomorphism from $G$ to $G_n$, and
    \item the subset $\theta_n(S)$ generates the group $G_n$
\end{itemize}
for all $n \in \mathbb{N}$.
\end{definition}

The first two assumptions in Definition \ref{def:essential LEF} are quite standard, and are closely related to the splitting assumption in \cite{KP}. Note that when all maps $\theta_n$ are group homomorphisms, namely when the LEF approximation arises via residual finiteness,  then the diagonal product $\bigotimes G_n$ is isomorphic to the group $G$ itself \cite[\S 4.4]{Mimura}. The third assumption can be  made without a loss of generality by restricting  the codomains of the functions $\theta_n$. All LEF approximations we will consider in this work will be essential and by simple groups.
This   assumption comes in handy in light of the following well-known group theory fact, see e.g. \cite[Lemma 4.4]{Mimura}.
\begin{lemma}\label{lem:subgrp_of_sum}
Let $G_i $ be a family of pairwise non-isomorphic simple groups for $i \in I$. Then any normal subgroup of the direct sum
$\bigoplus_{i \in I} G_i$ is of the form $\bigoplus_{i \in J} G_i$ for some subset $J \subset I$. 
\end{lemma}

\subsection*{The tail map}
Let $G$ be a group with a finite generating set $S$ and an \emph{essential} LEF approximation  $\theta_n : G \to G_n$ by \emph{simple} groups.  Consider the associated diagonal product $\Gamma = \bigotimes G_n$. The  diagonal product  $\Gamma$ comes equipped with a natural surjective homomorphism  $t : \Gamma \to G$ called the \emph{tail map}. 
%It will be important for our purposes.

The tail map $t$ is  defined on generators simply by $\gamma_s \mapsto s$ for each $s \in S$. To see  that $t$ is  a well-defined homomorphism, consider a word $w$ in the generators $\gamma_s$ such that $w(\gamma_s) = e$ in the group $\Gamma$. In particular $w(\theta_n(s)) = e$ in the finite group $G_n$ for all $n$ sufficiently large. This implies that $w(s) = e$ in the group $G$, as required. 
The tail map $t$ is clearly surjective, as  $S$ is  a generating set for $G$.

%We fix two important assumptions on the LEF approximation of $G$ throughout this section, which also appear in Theorem \ref{intro:thm:diag_prod_stab_main}. 

\begin{lemma}
\label{lem:sum_contained_in_diag_prod}
The direct sum $\bigoplus_n G_n$ is a subgroup of the diagonal product $\Gamma$.
The kernel of the tail map satisfies
    $  \ker t = \bigoplus_n G_n$.
\end{lemma}

\begin{proof}
The fact that  $\ker t = \Gamma \cap \bigoplus G_n$ follows readily from the definitions and holds  true for arbitrary diagonal products. It remains to show that the diagonal product $\Gamma$ contains the entire direct sum $\bigoplus G_n$. In light of Lemma \ref{lem:subgrp_of_sum}, it will be enough to show that the subgroup $\ker t$ surjects onto each simple factor $G_n$.
    
Consider the quotient map $\pi_n : \Gamma \to G_n$ for some $n \in \N$. The simplicity of the group $G_n$ implies that $\pi_n( \ker t)$ is either equal to $G_n$ or is trivial. Assume by contradiction that $\pi_n (\ker t) = \{e\}$.  Then  by \cite[Lemma 4.7]{KP} there exists a group homomorphism $p_n :G\to G_n$  such that $\pi_n = p_n \circ t$. It follows that $p_n$ agrees with $\theta_n$ on each generator $s\in S$, namely
    \[
        p_n(s)=p_n (t(\gamma_s))= \pi_n (\gamma_s)=\theta_n(s).
    \]
This contradicts the assumption that the LEF approximation $\theta_n$ is essential. 
\end{proof} 

In what follows we will use the notation $U = \ker t$.  The following  short exact sequence holds true
$$
1 \longrightarrow U \longrightarrow \Gamma \xlongrightarrow{t} G \longrightarrow 1.
$$
In particular, the diagonal product $\Gamma$ is an extension of the LEF group $G$ by the locally finite group $U$.
Consequently, if the LEF group $G$ is amenable, then so is the diagonal product $\Gamma$.
%See \cite{KP} for more details, and specifically Lemma 4.6 therein.

%It follows that any normal subgroup of $U$ is of the form $\bigoplus_{i \in L} G_i$ for some subset $L \subset \mathbb{N}$.

In the following lemma, we analyze finite index normal subgroups of $U$. 

\begin{lemma}
\label{lemma:algebraic structure of finite index subgroup of U}
Let $V \lhd U$ be a normal subgroup with $[U: V] < \infty$. Then $V \lhd \Gamma$ and there is a finite subset $L \subset \mathbb{N}$ such that
$$ \Gamma / V \cong G \times \prod_{n \in L} G_n.$$
%such that $V = \bigoplus_{n \in \mathbb{N} \setminus L} G_n$.
\end{lemma}
\begin{proof}
Recall that the diagonal product $\Gamma$ is by construction a subgroup of the direct product $\prod_n G_n$. It follows from the elementary group theoretical Lemma \ref{lem:subgrp_of_sum} that  there is some finite subset $L \subset \N$ such that
$$ V = \bigoplus _{n\notin L} G_n \quad \text{and} \quad U/V \cong \bigoplus_{n \in L} G_n.$$
This means that the subgroup  $V$ is normal in the direct product $\prod_n G_n$, and particularly in the diagonal product $\Gamma$. 

Let $p$ be the  quotient map 
$p : \Gamma \to  U/V$ obtained by identifying the quotient group $U/V$ with the product of the finitely many coordinates of $\prod_n G_n$ belonging to $L$.
Recall that the  tail map $t$ is the quotient map $t : \Gamma \to G$ with $\ker t = U$. The Cartesian product of these two quotient maps $p$ and $t$ is the homomorphism
\[
t\times p:  \Gamma \to G \times U / V .
\]
It is routine to verify that the product homomorphism $t \times p$ is surjective and has $\ker(t \times p) = V$. This sets up the desired group isomorphism. 
\end{proof}

\subsection*{The subgroups $U_n$ and $\Gamma_n$}

For each $n \in \mathbb{N}$ denote 
$$U_n = \prod_{i=1}^n G_i$$
so that $U_n$ is a finite  subgroup of $U$.
%Recall that for each $n \in \mathbb{N}$ we have denoted $U_n = \prod_{i=1}^n G_i$. 
%\ad{These two paragraphs should be moved to the subsection on algebraic structure of diagonal products}\iv{it used to be there, but then it seemed to adhok for the following proof. I am ok with both options}
Further, for each $n \in \mathbb{N}$  let $\Gamma_n$  be the subgroup of the direct product $\prod_{i} G_i$ generated by the elements 
$$ \gamma_{s,n} \in \prod_i G_i, \quad \gamma_{s,n}(i) = \begin{cases}
\theta_i(s) & i \ge n \\
e & 1 \le i < n
\end{cases}$$
for all $s \in S$.
In particular  $\Gamma_1 = \Gamma$. The fact that $U_n \le \Gamma$ implies that $\Gamma_{n+1} \le \Gamma$ as well. The subgroups $U_n$ and $\Gamma_n$ are all normal in $\Gamma$ and satisfy 
$$\Gamma = U_n \times \Gamma_{n+1} \quad \text{and} \quad \Gamma_n = G_n \times \Gamma_{n+1}$$
for all $n \in \mathbb{N}$.

Lastly,  there is a sequence of \enquote{asymptotic sections} to the tail map, namely, maps $$\Phi_n : G\to \Gamma_n$$ such that $t \circ \Phi_n(g) = g$ holds true for each element $g \in G$ and all $n$ sufficiently large. Indeed, for each pair of indices $n \le m$, the $m$-th coordinate of the map $\Phi_n$  is simply the  map $\theta_m : G \to G_m$ coming from the LEF approximation. It follows that the maps $\Phi_n$ are partial homomorphisms, in the sense of Definition \ref{def:partial homo and LEF approx}.

\subsection*{Locally inner actions and subgroups}

%We introduce locally inner actions. 

\begin{definition} \label{def:locally-inner}
    An automorphism $\sigma$ of a group $G$ is   \emph{locally inner} if for any finitely generated subgroup $F\leq G$ there is some  element $h\in G $ such that $\sigma(g)=g^h $ for all $g\in F$.
    %\ad{This seems stronger than what we actually need for our purposes, enough to talk about cyclic subgroups $H$, and not all finitely generated ones.}
    An action of a group $H$ on a group $G$ by
    automorphisms is   \emph{locally inner} if each element of $H$ is acting via a locally inner automorphism. A normal subgroup $N \lhd G$ is called a \emph{locally inner subgroup} if the conjugation action of the group $G$ on $N$ is locally inner.
\end{definition}

To clarify the terminology of a locally inner subgroup; the conjugation action of $G$ on its normal subgroup $N$ is required to be locally inner from the \enquote{point of view} of $N$. Further, note that being normal is part of the definition of a locally inner subgroup.
Here are some useful examples to keep in mind.

\begin{example} \label{exam:Sym_Alt_loc_inner}
The finitary groups $\SymZ$ and $\AltZ$ are locally inner  subgroups of the group $\Sym(\Z)$ consisting of \emph{all} permutations of the set $\Z$. 
\end{example}

\begin{example} \label{exam:prod_loc_inner}
Let $G_i$ be a countable family of groups. The direct sum $\bigoplus_i G_i$ is a locally inner normal subgroup of the direct product $\prod_i G_i$.
\end{example}

Certainly,   if some action of the group $G$   is locally inner, then the same is true for the restricted action of any subgroup of $G$.

We explore another  interesting example of locally inner subgroups that arises in the context of diagonal products.

\begin{lemma}
\label{lem:locally inner for diagonal products}
Let $G$ be a group with a finite generating set $S$ and an essential LEF approximation $\theta_n : G \to G_n$ by simple groups. Let $\Gamma = \bigotimes_n G_n$ be the associated diagonal product with tail map $t$. If  the subgroup $N \lhd G$ is  locally inner  then the subgroup $\Delta = t^{-1}(N) \lhd \Gamma$ is also locally inner.
\end{lemma}
\begin{proof}
Consider a pair of elements $\gamma \in \Gamma$ and $\delta \in \Delta$. The assumption that $N$ is locally inner in $G$ allows us to find some element $h\in N$ such that $t(\delta^\gamma) = t(\delta)^h$. Let $\delta_0 \in \Delta$ be an arbitrary element with $t(\delta_0) = h$. Then $t(\delta^\gamma) = t(\delta^{\delta_0})$ which means $\delta^\gamma = \delta^{\delta_0} u $ for some element $u \in U$. Let $n \in \mathbb{N}$ be some sufficiently large index such that $u \in U_n$. This means that there is some element $v \in U_n$ such that the element $\delta_1 = \delta_0 v \in \Delta$ has the same projection as the element $\gamma$  onto the first $n$-many  coordinates $\prod_{i=1}^n G_n$. We conclude that $\delta^{\delta_1} = \delta^\gamma$ as required.
\end{proof}

In particular, the subgroup $U = \ker t \lhd \Gamma$  is locally inner. This fact  also follows from Example \ref{exam:prod_loc_inner} discussed above.

\section{Classical and generalised B. H. Neumann groups}
\label{sec:generalised Neumann groups}
We discuss the classical groups of B.H. Neumann \cite{Neu} as well as their generalizations introduced by Bradford in \cite[\S 2.4]{Bradford_RF_growths}. 
Bradford's groups are quite   elementary yet interesting examples of diagonal products. They are closely related to the groups of fast residual finiteness growth introduced in \cite{BS}. 
We will show that Bradford's groups are Hilbert--Schmidt stable and have arbitrarily large stability radius growth  (Theorem \ref{intro:thm:arbitrary large SRad}). 
%We refer to \cite{Bradford_RF_growths} for a nice exposition of these examples.

\begin{definition}[Bradford \cite{Bradford_RF_growths}] \label{def:generalised B H Neumann}
Let $d, r: \N \to \N$ be a pair functions, such that the function $d$ is  strictly increasing  with $d(1) \ge 5$ and $d(n)$ odd for all $n$ and such that the function $r$ satisfies   $2r(n) \leq d(n) - 1$ for all $n$. Consider the two elements
\[\alpha, \beta  \in \prod_n \Alt(d(n)) \quad \text{where} \quad \alpha = (\alpha_n), \beta = (\beta_n) \]
given for each $n$ by
\[
\alpha_n = (1 \; 2 \; \dots \; d(n)) \in \Alt(d(n)) \quad \text{and} \quad \beta_n = (1 \; 1+ r(n) \; 1+ 2r(n)) \in \Alt(d(n)). \]
The \emph{$(d,r)$-generalized B. H. Neumann group} is  the subgroup the direct product $\prod_n \Alt(d(n))$ generated by the two elements $S = \{ \alpha, \beta \}$. It will be denoted $B(d,r)$.
\end{definition}

We will deal with generalised B. H.  Neumann groups in two regimes.

\begin{itemize}
    \item $r$ is the constant function $1$, so that $\beta_n = (1\;2\;3)$ for all $n$. 
    The group $B(d,1)$ is a classical   B. H. Neumann group \cite{Neu}. Two such groups $B(d,1)$ and $B(d',1)$ are  isomorphic if and only if  $d = d'$. This yields  an uncountable family of $2$-generated groups \cite{Neu}.
The group $B(d,1)$ is isomorphic to a diagonal product associated with a LEF approximation of the alternating enrichment $\Z \ltimes \Altfin{\Z}$ (this   will be elaborated  in Example \ref{example:LEF approximation giving rise to BS} below).

\item The function $r$ is unbounded.  
In this case, we will always be assuming that $d(n)$ is prime for all $n \in \N$ and that $$\lim_n r(n) = \lim_n  d(n) - 2r(n) = \infty.$$
Under these assumptions, Bradford \cite{Bradford_RF_growths} shows that the generalized B.H. Neumann group $B(d,r)$ is isomorphic to a diagonal product associated with a LEF approximation of the wreath product $(\Z / 3\Z) \wr \Z$.
\end{itemize}

Let us provide a bit more information about the generalized B.H. Neumann groups in the second regime on the above list.

\begin{proposition} \label{prop:Bradford groups are diag prod}
Consider the wreath product $G = (\Z / 3\Z) \wr \Z$ with the two generators $S = \{\alpha_\infty, \beta_\infty\}$ 
where $\alpha_\infty$ is a generator of the cyclic subgroup $\Z $ and $\beta_\infty$ is a generator of the  direct summand $\Z / 3\Z$ indexed by $0$ of $\bigoplus_\Z (\Z/3\Z)$.

Let $d: \N \to \N$ be a strictly increasing sequence of primes with $d(1) \ge 5$. Let $r: \N \to \N$ be a function such that
    $$ \lim_n  r(n) = \lim_n d(n) - 2r(n) = \infty.$$

Then there is an essential LEF approximation $\theta_n: G \to \Alt(d(n))$ by simple groups satisfying 
\[\theta_n(\alpha_\infty) = \alpha_n \quad \text{and} \quad \theta_n(\beta_\infty) = \beta_n\]  where the permutations $\alpha_n, \beta_n \in \Alt(d(n))$ are as in Definition \ref{def:generalised B H Neumann}, and 
\[ B(d,r) \cong \bigotimes_n \Alt(d(n))\]
where the diagonal product is taken with respect to the generating set $S$ and the LEF approximation $\theta_n$.
\end{proposition}
\begin{proof}
    This is essentially a reformulation  of \cite[Lemma 2.10 and Proposition 2.12]{Bradford_RF_growths}. These results establish that the Cayley graphs of $\Alt(d(m))$ with respect to the generators $\{\alpha_m, \beta_m\}$ converge in the Chabauty topology on the space of marked groups to the Cayley graph of $G$ with respect to the generators $\{\alpha_\infty, \beta_\infty \}$.
    This guarantees the existence of a LEF approximation $\theta_n: G \to \Alt(d(n))$, by the well-known connection between LEF approximations and limits of finite marked groups \cite[Theorem 2.16]{FFGS}.
    Hence $B(d,r)$ is the diagonal product to this LEF approximation.
    Further, the restrictions of $\theta_n$ to the generating pairs ${\{\alpha_\infty, \beta_\infty\}}$ do not extend to homomorphisms from $G$ to $\Alt(d(n))$, since the group $G$ is metabelian, while the groups $\Alt(d(n))$ are simple and non-abelian.
%    Finally, previous work of the last two authors shows that $G = \Z/3 \wr \Z$ is Hilbert--Schmidt stable \cite{levit2023characters}[Corollary D], so the assumptions of Theorem \ref{intro:thm:diag_prod_stab_main} are satisfied.
\end{proof}

\section{Alternating and elementary enrichments and diagonal products}
\label{sec:diagonal products from alternating and elementary enrichments}

Let $G$ be a countable group.  We study two particular extensions of the group $G$, namely the alternating and the elementary enrichments. We also describe a two-step procedure we call \emph{alternating (or elementary) diagonal products}. It is inspired by Mimura's work \cite{Mimura}. Applying this procedure to the infinite cyclic group $\Z$ give a (somewhat roundabout) construction of the classical B.H. Neumann groups.

\subsection*{Enrichments}
Let $\Altfin{G}$ denote the group of all finitely-supported even permutations of the underlying set of the group $G$. 

\begin{definition}
\label{def:alternating enrichment}
The \emph{alternating enrichment} of a group $G$ is the semidirect product $\mathscr{A}(G) = G \ltimes \Altfin{G}$.   The action of the group $G$ by automorphisms on the normal subgroup $\Altfin{G}$ arises from conjugation by the left action of $G$ on its underlying set. 
\end{definition}

Fix a finite field $F = \mathbb{F}_{p^k}$ for some prime $p$ and some $k \in \mathbb{N}$. Given a set $X$  (for us, this will always be the underlying set of a group) we will denote by $V_X$ be the $F$-vector space with basis indexed by $X$.  Let  $\SLfin{V_G}$ be the group of all unimodular finite rank $F$-linear transformations of the vector space $V_G$.

\begin{definition}
\label{def:elementary enrichment}
The \emph{elementary enrichment} of $G$ is the semidirect product $\mathscr E(G) = G \ltimes \SLfin{V_G}$.   The action of the group $G$ by automorphisms on $\SLfin{V_G}$ arises  from conjugation by the left-regular representation.  
\end{definition}

\subsection*{Alternating enrichments of LEF groups} 

Assume that the countable group $G$ has  a fixed finite generating set $S$ and a  LEF approximation $\theta_n :G \to G_n$ such that the orders $|G_n|$ are pairwise distinct  (see Definition \ref{def:partial homo and LEF approx}). 

In addition, we make a technical parity assumption.  Namely, we assume that the regular permutation action of the element  $\theta_n(s)$ on the underlying set of the group $G_n$ is even for each $s \in S$ and $n \in \mathbb{N}$. This assumption is equivalent to requiring that either the order of $\theta_n(s)$ is odd or $\left[G_n: \left<\theta_n(s)\right>\right] $ is even.

\begin{lemma}[Encoding into alternating groups]
\label{lemma:encoding into alternating}
The alternating enrichment ${\mathscr A(G)}$ is finitely generated and admits an essential LEF approximation $\Theta_n : \mathscr{A}(G) \to \Alt(G_n)$ by simple groups.
\end{lemma}
\begin{proof}
Consider the alternating enrichment  ${\mathscr A(G)} = G \ltimes \Altfin{G}$  of the group $G$.
The group ${\mathscr A(G)}$ is finitely generated. Indeed, for each generator $s \in S$ consider the $3$-cycle $t_s \in \Altfin{G}$ given by the $3$-cycle $t_s = (s^{-1},e,s)$. Denote $T = \{t_s \: : \: s \in S\}$. It is not hard to verify that $S \amalg T$ is a finite generating set for ${\mathscr A(G)}$. One may rely on the fact that $\Altfin{G}$ is generated by all $3$-cycles of the form $(g_1,g_2,h)$ where $g_1,g_2 \in G$ are fixed and $h \in G \setminus \{g_1,g_2\}$ is arbitrary.

The group ${\mathscr A(G)}$ is LEF. It admits a LEF approximation $\Theta_n : \mathscr{A}(G) \to \Alt(G_n)$ called \enquote{encoding into alternating groups} and described in \cite[\S4]{Mimura}. It is obtained by letting the element $\Theta_n(s)$ be the permutation given by left multiplication by $\theta_n(s)$, and the element $\Theta_n(t_s)$ be the $3$-cycle $(\theta_n(s^{-1}),\theta_n(e),\theta_n(s))$ for all $n$.
\end{proof}

\begin{definition}
\label{def:enriched alternating diagonal product}
The \emph{alternating diagonal product} associated to the group $G$ with its finite generating set $S$ and its LEF approximation $\theta_n : G \to G_n$  is the diagonal product $\Gamma = \bigotimes \Alt(G_n)$ associated to the alternating enrichment   ${\mathscr A(G)}$ with its generating set $S \amalg T$ and  its LEF approximation $\Theta_n : \mathscr{A}(G) \to \Alt(G_n)$ defined by \enquote{encoding into alternating groups}.
\end{definition}

We are ready to present the classical groups of B.H. Neumann \cite{Neu}.

\begin{example}[Classical B. H. Neumann groups]
\label{example:LEF approximation giving rise to BS}
Consider the infinite cyclic group $\Z$ with its generating set $S = \{1\}$. The corresponding alternating enrichment is the group ${\mathscr A(G)} = \Z \ltimes \Altfin{\Z}$.
Let $\mathcal{P} = (d_n)$ be some strictly increasing sequence of odd integers with $d_1 \ge 5$. Consider the sequence of finite cyclic groups $\Z/d_n \Z$ and the LEF approximation $\theta_n : \Z  \to \Z / d_n \Z$ given simply by   the residue class modulo $d_n$. Note that $\Theta_n(1)$ is the full $d_n$-cycle  and $\Theta(t_1)$ is the $3$-cycle $(-1,0,1)$ for all $n$. We conclude that the alternating diagonal product associated to this data is the classical B.H. Neumann group associated to sequence $\mathcal{P}$. See \cite{LPS} or \cite{LL2} for a more hands-on description of these groups.
\end{example}

B.H. Neumann groups are revisited in \S\ref{sec:generalised Neumann groups} below, where the classical groups constructed in Example \ref{example:LEF approximation giving rise to BS} are denoted
$B(d,1)$.

\subsection*{Elementary enrichments of LEF groups}
Assume that the countable group $G$ has  a fixed finite generating set $S$ and a  LEF approximation $\theta_n :G \to G_n$ such that the orders $|G_n|$ are pairwise distinct and such that the above-mentioned parity assumption holds true.

\begin{lemma}[Encoding into elementary groups]
\label{lemma:encoding into elementary}
The elementary enrichment ${\mathscr E(G)}$ is finitely generated and admits an essential LEF approximation $\Theta_n : \mathscr{E}(G) \to \SL(V_{G_n})$ by simple groups.
\end{lemma}
\begin{proof}
Consider the elementary enrichment ${\mathscr E(G)} = G \ltimes \SLfin{V_G}$ of the group $G$.
The group ${\mathscr E(G)}$ is finitely generated. Indeed, consider the finite set $T = \{t_s^x \: : \: s \in S, x \in F \setminus \{0\} \}$ where each $t_s^x \in \SLfin{V_G}$ is the elementary linear transformation adding an $x$-th multiple of the $e$-th coordinate to the $s$-th coordinate. It is routine to verify that $S \amalg T$ is a finite generating set for ${\mathscr E(G)}$.

The group ${\mathscr E(G)}$ is LEF. It admits a LEF approximation $\Theta_n : \mathscr{E} \to \mathrm{SL}(V_{G_n})$ called \enquote{encoding into elementary groups} and described in \cite[\S6]{Mimura}. 
It is defined so that $\Theta_n(s)$ is the permutation matrix corresponding to right multiplication by $\theta_n(s)$, and  $\Theta_n(t_s^x)$ is the elementary linear transformation obtained by adding an $x$-th multiple of the $\theta_n(e)$-coordinate to the $\theta_n(s)$-coordinate.
\end{proof}

\begin{definition}
\label{def:elementary enriched diagonal product}
The \emph{elementary diagonal product} associated to the group $G$ with its finite generating set $S$ and its LEF approximation  $\theta_n : G \to G_n$ is the diagonal product $\Gamma = \bigotimes \mathrm{SL}(V_{G_n})$  associated to the elementary enrichment   ${\mathscr E(G)}$ with its generating set $S \amalg T$ and its LEF approximation $\Theta_n : \mathscr{E}(G) \to \mathrm{SL}(V_{G_n})$ defined by \enquote{encoding into elementary groups}.
\end{definition}

%With the above data, let $\Gamma = \bigotimes (\mathrm{SL}(V_{G_i}), S_i \amalg T_i)$ be the associated diagonal product. We call $\Gamma$ the \emph{enriched elementary diagonal product} associated to the original LEF approximation $G_i$ to the group $G$.

%Note that $\Gamma_0 \cong G \ltimes \mathrm{SL}(V_G)  \cong G \ltimes N$ where we denote $N = \mathrm{SL}(V_G)$. The group $N$ is locally finite.  In our situation, we have maps $h_i : \mathrm{SL}(V_{G_i}) \to N$ which can be regarded simply as inclusion maps.  

\begin{example}
\label{exam:LEF approximation elementary}
Consider the infinite cyclic group $\Z$ with its generating set $S = \{1\}$ and its LEF approximation $\theta_n : \Z \to \Z/d_n \Z$ exactly as in Example \ref{example:LEF approximation giving rise to BS}. The associated elementary diagonal product $\bigotimes \mathrm{SL}_{d_n}(F)$ appears in \cite[Example 4.2]{LW}.
\end{example}

\mypart{Character theory}

\section{Characters  of discrete groups}\label{sec:characters}

%\subsection*{Generalities}\label{subsec:generalities}
Let $G$ be a countable group.

\begin{definition}
\label{def:trace}
A \emph{trace} on the group $G$ is a function $\varphi:G\to\C$ satisfying
the following properties:
\begin{enumerate}
\item $\varphi$ is \emph{positive-definite}: for all $n\in\N$ and any choice of group elements $g_{1},...,g_{n}\in G$ and
 scalars $\alpha_{1},...,\alpha_{n}\in\C$ we have
$\sum_{i,j=1}^n\alpha_{i}\overline{\alpha}_{j}\varphi(g_{j}^{-1}g_{i})\geq0$.

\item $\varphi$ is \emph{normalized}: $\varphi(e)=1$. 
\item $\varphi$ is \emph{conjugation-invariant}: $\varphi(g^h)=\varphi(g)$
for any pair of elements $g,h\in G$.
\end{enumerate}
\end{definition}
The space of all traces on the group $G$ is denoted by $\Tr G$. We endow $\Tr G$ with the topology of pointwise convergence. Any trace $\varphi$ on $G$ is bounded, and in fact
$\|\varphi\|_\infty = \varphi(e) = 1$.
%This is the only topology
%on $\Tr G$ we shall consider in this paper. 
It follows  that $\Tr G$ is a convex and compact  subset of the unit ball in $\ell^{\infty}(G)$  endowed with the pointwise convergence topology. As the group $G$ is  countable, the space $\Tr G$ is  metrizable.

The extreme points of the convex set $\Tr G$ are called \emph{characters}. That is, characters are traces  that cannot be written
as a proper convex combination of traces.
We denote by $\Ch G$ the space of  characters of the group $G$. 

Given a Borel probability measure $\mu$ on $\Tr G$ its 
 barycenter is the trace $\varphi$ given by
\[
\varphi = \int_{\Tr G} \psi \; \mathrm{d} \mu(\psi).
\]
The convexity and  compactness of the set $\Tr{G}$ ensures that $\varphi \in \Tr G$ \cite[\S 1]{phelps2001lectures}. Any trace $\varphi \in \Tr G$  arises as the barycenter of a unique Borel probability measure $\mu_\varphi$ on $\Tr{G}$ satisfying $\mu_\varphi(\Ch{G})=1$, see \cite{Thoma-characters}.

Extending the representation theory of finite groups, finite-dimensional unitary representations give rise to traces. Indeed, let $\pi$ be a finite-dimensional unitary representation of the group
$G$. Then the function $\varphi_\pi =\frac{1}{\dim\pi}\tr \circ \pi(g)$ is a trace on $G$.  We call such traces \emph{finite-dimensional}. 
%These traces are  considered in the statement of Theorem \ref{thm:hadwin-shulman}.
See  \cite[\S9]{levit2023characters} for more information on finite-dimensional traces.

Associated to any trace $\varphi \in \Tr G$ is the \emph{Gelfand–Naimark–Segal (GNS) construction}. It consists of a Hilbert space $\mathcal{H}_\varphi$, a unitary representation $\pi_\varphi : G \to \mathcal{U}(\mathcal{H}_\varphi)$ and a cyclic unit vector $v_\varphi \in \mathcal{H}_\varphi$ such that $\varphi(g) = \left<\pi_\varphi(g) v_\varphi, v_\varphi\right>$ for all elements $g \in G$. The GNS data consisting of the tuple $(\mathcal{H}_\varphi,\pi_\varphi,v_\varphi)$ is uniquely determined up to an isomorphism. The Hilbert space $\mathcal{H}_\varphi$ is finite-dimensional if and only if the trace $\varphi$ is finite-dimensional (in the above sense).

%\begin{proposition}
%\end{proposition}

%   In the realm of amenable groups, there is a very interesting connection
%   between stability and characters. We recall that a discrete group
%   $G$ is called \emph{amenable} if there exists a bounded linear functional
%   $m$ (called a mean) on the Banach space $\ell^{\infty}(G)$, such that
%   $m(1_{G})=1$, and such that $m(\psi)=m(g.\psi)$ for any $g\in G$
%   and $\psi\in \ell^{\infty}(G)$. 

%   \begin{theorem}[Hadwin-Shulman]\label{thm:hadwin-shulman}
%    An amenable group $G$ is Hilbert--Schmidt stable if and only if
%   any trace on $G$ is a limit of finite dimensional traces. 
%   \end{theorem}
% \begin{theorem}[\cite{thoma1964unitare}]\label{thm:choquet}
% Let $G$ be a countable group. Then any trace $\varphi$ on $G$
% is the barycenter of a unique Borel probability measure $\mu$ on
% $\Tr G$ with $\mu(\Ch G)=1$.
% \end{theorem}

%\subsection*{The kernel of a trace}

\begin{definition}
The \emph{kernel} of the trace $\varphi  \in \Tr G$ is
\[
\ker\varphi=\left\{ g\in G \: : \: \varphi(g)=1\right\}.
\]
The trace $\varphi$ is called \emph{faithful} if $\ker\varphi=\{e\}$. $\varphi$ is called \emph{trivial} if $\ker \varphi = G$ so that $\varphi \equiv 1$. 
\end{definition}

The kernel  of any trace $\varphi \in \Tr G$ is a normal subgroup and $\varphi$ factorizes to a faithful trace on the quotient group $G/\ker\varphi$. For this see e.g. \cite{levit2023characters}. 
 
\begin{lemma}
\label{lemma:traces of abs value 1}
Let $\varphi \in \Tr{G}$ be a trace. If $|\varphi(g) |=1$  for all $g\in G$  then $\varphi : G \to \mathbb{C}^*$ is a multiplicative character.
\end{lemma}
\begin{proof}
Let $(\mathcal{H}_\varphi,\pi_\varphi,v_\varphi)$ be the GNS data associated to the trace $\varphi$. Note that $v_\varphi$ is a unit vector, namely $\|v_\varphi\| = 1$.
%This means that $\mathcal{H}_\varphi$ is a Hilbert space, $\pi_\varphi $ is a unitary representation of the group $G$ on $\mathcal{H}_\varphi$ and $v_\varphi \in \mathcal{H}_\varphi$ is a cyclic unit vector satisfying $\varphi(g) = \left<\pi_\varphi(g)v_\varphi,v_\varphi\right>$ for all elements $g \in G$. 
Note that  $|\varphi(g)| = 1$ for some element $g \in G$ if and only if  $|\left<\pi_\varphi(g)v_\varphi,v_\varphi\right>| = 1$. In that case, the equality condition in the Cauchy--Schwarz inequality implies that  $\pi_\varphi(g) v_\varphi = \lambda_g v_\varphi$ for some scalar $\lambda_\varphi \in \mathbb{C}$ with $|\lambda_g| = 1$. As the vector $v_\varphi$ is cyclic, if $|\varphi(g)| = 1$ holds for all elements $g \in G$ then $\mathcal{H}_\varphi = \mathrm{span}_\mathbb{C} v_\varphi$ and $\pi_\varphi(g) = \lambda_g$ for all $g \in G$. The map $g \mapsto \lambda_g$ is a multiplicative homomorphism.
\end{proof}

The following is well known, and we skip it's proof.

\begin{proposition}\label{prop:characters-of-quotients}
%\marginpar{do we ever use Proposition \ref{prop:characters-of-quotients}, implictly or explicitly?}
Let $f:G\to H$ be a quotient map of groups. The pullback map $f^{*}:\Tr H\to\Tr G$
defined by $f^{*}(\psi)=\psi\circ f$ is injective, continuous and
affine. The image of $f^{*}$ is  a face of $\Tr G$.
In particular, a trace $\psi \in \Tr{H}$ is a character of $H$ if and only if $f^{*}(\psi)$
is a character of $G$. 
\end{proposition}

Recall that a \emph{face} of a convex set $C$ is a convex subset $F \subset C$ such that for every pair of points $x, y \in C$, if $(1-\lambda)x + \lambda y \in F$ for some $\lambda \in (0,1)$ then $x,y \in F$.

%   Let $\psi$ be a character on $H$ and denote $\varphi=f^{*}(\psi)$.
%   Consider a convex combination of the form $\varphi=\lambda\varphi_{1}+(1-\lambda)\varphi_{2}$
%   for $\varphi_{1},\varphi_{2}\in\Tr G$ and $0<\lambda<1$. Let $g\in\ker f$.
%   Then in particular $g\in\ker\varphi$ and so $\lambda\varphi_{1}(g)+(1-\lambda)\varphi_{2}(g)=1$.
%   Since traces attain values in the unit disk in $\C$, then $\varphi_{1}(g)=\varphi_{2}(g)=1$.
%   It follows that the kernels of $\varphi_{1}$ and $\varphi_{2}$ contain
%   $\ker f$, so that by the above, there exists traces $\psi_{1}$ and
%   $\psi_{2}$ on $G$ such that $f^{*}(\psi_{1})=\varphi_{1}$ and $f^{*}(\psi_{2})=\varphi_{2}$.
%   Since $f^{*}$ is affine, we have that $\psi=\lambda\psi_{1}+(1-\lambda)\psi_{2}$.
%   But $\psi$ is a character, hence $\psi_{1}=\psi_{2}$ and so $\varphi_{1}=\varphi_{2}$,
%   and we got that $\varphi$ is a character. The other direction follows
%   from $f$ being affine and injective.

\subsection*{Restriction and induction of traces}
Consider a normal subgroup  $N$ of the group $G$. 
Denote
\[\Tr N^G =  \{ \psi \in \Tr N \: : \: \psi(n^g) = \psi(n) \quad \forall n\in N, g \in G\}. \] 
For any trace $\varphi \in \Tr G$ its restriction $\varphi_{\mid N}$ belongs to $\Tr G^N$. Namely, restriction  gives rise to a continuous affine map $\Tr G\to\Tr N ^G$. 

Let $\psi \in \Tr{N}$ be any trace. The \emph{trivial extension} of $\psi$ to the group $G$ is given by
\[
\widetilde{\psi}:G\to\C, \quad g\mapsto\begin{cases}
\psi(g) & \forall g\in N\\
0 & \forall g\notin N
\end{cases}.
\]
It is easy to see that if $\varphi \in \Tr N^G$ then its trivial extension satisfies $\widetilde{\varphi} \in \Tr{G}$.  We get  a continuous affine injective
map $\Tr N^{G}\to\Tr G$. 
In terms of the GNS construction, the trivial extension corresponds to induction of unitary representations from the subgroup $N$ to the group $G$. This discussion shows that the restriction map $\Tr G\to\Tr N ^G$ is surjective.
%We  say that the trace $\widetilde{\psi}$ is \emph{induced} from the trace  $\psi$. 

For example, the trivial extension $\widetilde{1}_N $ of  the trivial character $1_N \in \Tr{N}$ is the characteristic function of the normal subgroup $N$.  In particular, the Dirac function at the identity element $\delta_{e}$ is always a trace. 
The trivial extension $\widetilde{1}_N $ is a character of the group $G$ if and only if the quotient group $Q=G/N$ has the \emph{ICC property}, i.e every non-trivial conjugacy class of  $Q$ is infinite.

In general, restrictions of characters to normal subgroups as well as trivial extensions of
characters   do  not result in characters. Nevertheless, in many instances considered in this paper, the resulting trace turns out to be a character. See \S\ref{sec:locally inner}.

\subsection*{Characters of direct sums}
Let $I$ be a countable index set. Let $G_i$ be a countable group for each index $ i \in I$. Consider the direct sum group $G = \bigoplus_{i \in I} G_i$. 

\begin{definition}
Let $\varphi_i \in \Tr {G_i}$ be a trace for each index $i \in I$. The  \emph{tensor product} of the $\varphi_i$'s is given by
\[
        \varphi=\otimes_{i\in I}\varphi_i :G \to \C, \quad \varphi : (g_i)\mapsto \prod_{i\in I}\varphi_i(g_i) \quad \forall (g_i)_{i \in I} \in G.
\]
\end{definition}
This infinite  product is well defined since all but finitely many multiplicands are equal to $1$ for each given element of $G$.
It is not hard to see that the tensor product $\varphi = \otimes_{i \in I} \varphi_i$ is indeed a trace on the group $G$. 

A description of the characters of the group $G$ in terms of tensor products of  characters of the direct summands $G_i$ was obtained by Thoma \cite[Satz 5]{Thoma-characters}. 

\begin{proposition}[Thoma] \label{prop:char_direct_prod}
Tensor product of traces sets up a bijection $\Ch{G} \cong \prod_{i\in I} \Ch{G_i}$.
\end{proposition}

For the sake of completeness, we provide a short proof of Thoma's result, which is based on the special case of products of finitely many groups covered in  \cite[Theorem 2.12]{BF}.

\begin{proof}[Proof of Proposition \ref{prop:char_direct_prod}]
For each subset $J\subseteq I$ it will be convenient to denote $G_J=\bigoplus_{i\in J} G_i$.
Consider the tensor product map
\[F : \prod_{i \in I} \Ch{G_i} \to \Tr{G}, \quad F((\varphi_i)_{i \in I}) = \otimes_{i\in I} \varphi_i.\]
The map $F$ is clearly injective, for the restriction of the trace $F((\varphi_i)_{i \in I})$ to each summand $G_i$ recovers the trace $\varphi_i$. 
Let us verify that the image of the map $F$ is $\Ch{G}$. Consider an arbitrary character $\varphi \in \Ch G$. The restriction $\varphi_i = \varphi _{\mid G_i}$ belongs to $\Ch {G_i}$, see e.g. Proposition \ref{prop:locally-inner} below. By \cite[Theorem 2.12]{BF} it is the case that the restriction of $\varphi$ to the subgroup $\bigoplus_{i \in J} G_i$ is equal to $\otimes_{i\in J} \varphi_i$ for each finite collection of indices $J \subset I$. Therefore $F((\varphi_i)_{i \in I}) = \varphi$.

It remains to show  that the tensor product of characters is a character. Let $\varphi_i \in \Ch{G_i}$ be an arbitrary character for each index $i \in I$. The case of finitely many summands  considered in \cite[Theorem 2.12]{BF} implies that the restriction of the trace $F((\varphi_i)_{i \in I})$ to the subgroup $\bigoplus_{i\in J} G_i$ is a character for each finite collection of indices $J \subset I$. It follows that $F((\varphi_i)_{i \in I}) \in \Ch{G}$ as required.
\end{proof}

\begin{corollary}
\label{cor:product restricting to char on one}
Let $G = G_1 \times G_2$ be a direct product. Let $\varphi \in \Tr{G}$ be a trace with restrictions $\varphi_i = \varphi_{\mid G_i} \in \Tr{G_i}$  for $i \in \{1,2\}$. If $\varphi_1 \in \Ch{G_1}$ then $\varphi = \varphi_1 \cdot \varphi_2$.
\end{corollary}
\begin{proof}
Let $\mu_\varphi$ be the unique Borel probability measure on $\Ch{G}$ such that $\varphi = \int \psi \; \mathrm{d} \mu_\varphi(\psi)$.
We know that $\Ch{G} = \Ch{G_1} \times \Ch{G_2}$ by Proposition \ref{prop:char_direct_prod}. Since $\varphi_1$ is a character, the push-forward of the probability measure $\mu_\varphi$ to the factor $\Ch{G_1}$ is the Dirac point mass $\delta_{\varphi_1}$. The push-forward of $\mu_\varphi$ to the other factor $\Ch{G_2}$ is some Borel probability measure $\mu_2$ satisfying $\varphi_2 = \int \psi' \; \mathrm{d} \mu_2(\psi')$. By measure-theoretical reasons, it must be the case that $\mu = \delta_{\varphi_1} \otimes \mu_2$. This is equivalent to the desired conclusion.
\end{proof}

\begin{proposition}
\label{prop:prod fin dim traces}
If  $\varphi_1, \varphi_2 \in \Tr{G}$ then $\varphi_1 \cdot \varphi_2 \in \Tr{G}$. Further, if the two traces $\varphi_1$ and $\varphi_2$ are finite-dimensional, then so is the trace $\varphi_1\cdot\varphi_2$.
\end{proposition}
\begin{proof}
The function   $  \varphi_1 \cdot \varphi_2$   is   conjugation-invariant  and satisfies $(\varphi_1 \cdot \varphi_2)(e)=1$. Further, the function $\varphi_1 \cdot \varphi_2$ is positive-definite, see e.g. \cite[Proposition C.1.6]{bekka2014kazhdan}. It follows that $\varphi_1 \cdot \varphi_2$ is a trace on the group $G$.

Now, assume that both traces $\varphi_1$  and $\varphi_2$ are finite-dimensional, namely $\varphi_i=\frac{1}{\dim\pi_i}\mathrm{tr}\circ \pi _i$ for some pair of finite-dimensional unitary representations $\pi_1$ and $\pi_2$   of $G$. Then $\varphi_1 \cdot \varphi_2$ is also finite-dimensional, for $\varphi_1 \cdot \varphi_2 =\frac{1}{\dim \pi}\mathrm{tr}\circ \pi $ where $\pi=\pi_1\otimes\pi_2$ is the tensor product representation.
\end{proof}

%Another fact we will use is that one can take a product of traces and end up with a trace on the same group:

\subsection*{Bekka's vanishing lemma for asymptotically orthogonal sequences}

A very useful  tool in the study of characters  is a vanishing lemma of Bekka  \cite[Lemma 16]{Bekka}; see  \cite[Lemma 4.14]{lavi2023characters} for a relative version. 
It  is used to show that a character which restricts to the Dirac character on a large subgroup must  vanish outside of that subgroup.
Here we present a generalization dealing with characters that \emph{asymptotically} vanish. 

\begin{lemma}
\label{lemma:asymptotically normal implies weakly to 0}
Let $\mathcal{H}$ be a Hilbert space and  $(\xi_n)_{n \in \mathbb N} \in \mathcal{H}$ be a sequence of unit vectors. If  $\lim_{n \to \infty} \langle \xi_n, \xi_m \rangle = 0$ for each  $m \in \mathbb N$ then the sequence $\xi_n$ weakly converges to $0$.
\end{lemma}
\begin{proof}
Fix an arbitrary vector  $\eta \in \mathcal H$. We are required to  show  that $\lim_n \langle \xi_n, \eta \rangle = 0$. We may assume without loss of generality that 
$\eta$ is in the closure of the span of the vectors  $\xi_n$ by replacing $\eta$ with its  orthogonal projection to that subspace.
For each  $\varepsilon > 0$ there is some  $N \in \mathbb N$ and values $a_1,\ldots,a_N \in \mathbb C$ such that $\| \sum_{i=1}^N a_i \xi_i - \eta \| 
< \varepsilon$. This means that 
$$ \limsup_n | \langle \xi_n, \eta \rangle | \leq \limsup_n | \langle \xi_n, \sum_{i=1}^N a_i \xi_i \rangle | + \varepsilon =\varepsilon.$$
The desired conclusion  follows.
\end{proof}

\begin{lemma}[Generalised Bekka lemma] \label{lem:gen_bekka}
%   Let $G$ be a group, $N \trianglelefteq G$ be a normal subgroup and $\varphi \in \Tr{G}$. Fix $g \in G$ and assume there is a sequence $x_n \in N$
%   such that $\lim_{|n-m| \to \infty}\varphi([g, x_n]^{-1}[g, x_m]) = 0$. Then $\varphi(g) = 0$.
Consider a trace  $\varphi \in \Tr{G}$. Fix an element $g \in G$. If there is a sequence of elements $x_n \in G$ such that
$$\lim_{n \to \infty}\varphi([g, x_n]^{-1}[g, x_m]) = 0$$ for every $m \in \mathbb N$ then $\varphi(g) = 0$.
\end{lemma}
\begin{proof}
Let $(\mathcal{H}_\varphi,\pi_\varphi,v_\varphi)$ be  the GNS data  associated to the trace $\varphi$. Namely $\mathcal{H}_\varphi$ is a Hilbert space, $\pi_\varphi : G \to \mathcal{U}(\mathcal{H}_\varphi)$ is a unitary representation and $v_\varphi \in \mathcal{H}_\varphi$ is a cyclic vector satisfying
\[
    \varphi(g) = \langle \pi_\varphi(g)v_\varphi,v_\varphi \rangle \quad \forall g\in G.
\]
It follows that the vectors $\xi_n = \pi_\varphi(\left[g,x_n\right]) v_\varphi \in \mathcal{H}_\varphi$  satisfy
$$ \langle \xi_n, \xi_m \rangle = \langle \pi([g, x_m]) \xi, \pi([g, x_n]) \xi \rangle  = \varphi([g, x_n]^{-1} [g, x_m]) \quad \forall n,m\in\N.$$
The assumption of this lemma together with Lemma \ref{lemma:asymptotically normal implies weakly to 0} imply that the sequence $\xi_n$   converges to $0$ in the weak topology. Consequently
\begin{align*}
\varphi(g)= \lim_n\varphi(x_n^{-1} g x_n)=\lim_n \varphi(g[g,x_n]) &=\lim_n \langle \pi_\varphi(g[g, x_n])v_\varphi, v_\varphi \rangle \\ &= \lim_n \langle \xi_n , \pi_\varphi(g^{-1}) v_\varphi \rangle = 0.
\end{align*}
\end{proof}

We will typically apply Lemma \ref{lem:gen_bekka} in the situation where the elements $x_n$ belong to some normal subgroup $N \lhd G$. In that case, the product of commutators $[g,x_n]^{-1}[g,x_m]$ will belong to the subgroup $N$ as well. Thus Lemma \ref{lem:gen_bekka} recovers  Bekka's vanishing lemma in the special case where the restriction of the trace $\varphi$ to the normal subgroup $N$ is   Dirac. 

% \subsection*{Rank functions}

% \begin{definition}
% Assume that there is a \emph{rank function} $\rank : G \to \mathbb{N}$. Which is constant on conjugacy classes. And satisfies that for every fixed non-trivial character $\varphi \in \Ch{G}$ we have
% $$ \limsup_n \{ |\varphi(g)| \: : \: \rank(g) = n\} = 0.$$
% So, elements with large rank, have small character values. We call such a group \emph{ranked}.
% \end{definition}

% \begin{lemma}
% Let $N \lhd G$ be a ranked normal subgroup. Assume that every element $g \in G \setminus N$ there is a sequence of elements $x_n \in N$ such that for every $m$
% $$ \lim \rank ( \left[g,x_n\right]^{-1} \left[g,x_m\right]) = \infty. $$
% Then the subgroup $N$ is inducing.
% \end{lemma}
% \begin{proof}
% Follows from the definitions and the generalized Bekka lemma.
% \end{proof}

\section{Locally inner and inducing subgroups}
\label{sec:locally inner}

We have introduced locally inner actions and subgroups in \S\ref{sec:LEF groups and diagonal products}. These actions are \enquote{degenerate} from the point of view of character theory and make induction and restriction of traces easier to study. Here is why.

%The operations of restriction and induction of traces and characters are particularly well behaved for locally inner  subgroups.
%, and give a complete classification in the case where the subgroup is inducing as we shall see.

\begin{proposition}\label{prop:locally-inner}
Let $G$ be a countable group and $N$ a locally inner  subgroup of $G$.  Then
\begin{enumerate}
    \item $\Tr{N}^G = \Tr{N}$.
    \item If $\varphi \in \Ch{G}$ then $\varphi_{\mid N} \in \Ch{N}$.
    \item If $\psi \in \Tr{N}$ then $\widetilde{\psi} \in \Tr{G}$.
\end{enumerate}
\end{proposition}
\begin{proof}
The conjugation action on the group $G$ on its normal subgroup $N$ induces an   affine action of the group $G$ on the compact convex set $\Tr{N}$. We claim that this action is trivial. Indeed, consider some trace $\psi \in \Tr{N}$ and some element $g \in G$. The automorphism of $N$ corresponding to conjugation by $g$ is locally inner. So for every element $n \in N$ there is some element $m \in N$ such that $n^g = n^m$. Hence $\varphi^g(n) = \varphi(n^g) = \varphi(n^m) = \varphi(n)$. We get   $\varphi \in \Tr{N}^G$. Item (1)  follows.

To prove $(2)$ consider a character $\varphi \in \Ch G$. This means that the restriction $\varphi_{\mid N}$ is an extreme point of the compact convex set $\Tr{N}^G$ of relative traces  \cite[Theorem 2.11]{BF}. However $\Tr{N}^G = \Tr{N}$ by  (1). Therefore $\varphi_{\mid N} \in \Ch{N}$ as required. 
 
We turn to proving  $(3)$. Consider a trace $\psi \in \Tr{N}$. Its trivial extension $\widetilde{\psi}$ is normalized and positive definite; see \cite[Proposition 1.F.9]{BdlH}. From the fact that $\Tr{N} = \Tr{N}^G$, it is immediate   that $\widetilde{\psi}$ is conjugation invariant as a function on $G$. So $\psi \in \Tr{G}$.
\end{proof}

% \begin{corollary}\label{cor:char_restriction_U}
% The restrictions of a character of $G(\mathcal P)$ to $N(\mathcal P)$ and $U(\mathcal P)$ are again characters.
% For $\psi \in \Tr{N(\mathcal P)}$, the induced trace, given by the trivial extension $\tilde{\psi}$, is a trace on $G(\mathcal P)$.
% \end{corollary}

\subsection*{Inducing subgroups}
The following notion captures a vanishing phenomenon which will play a role throughout our analysis.

\begin{definition}
\label{def:inducing and weakly inducing}
A normal subgroup $N \lhd G$ is called \emph{inducing} (respectively  \emph{weakly inducing})   if  every character $\varphi \in \Ch{G}$ such that $N \cap \ker \varphi \lneq N$ (respectively $\left[N: N \cap \ker \varphi\right] = \infty$) satisfies $\varphi(g) = 0 $ for all $g \in G \setminus N$.
\end{definition}

Note that being inducing is a relative notion, which depends on the normal subgroup $N$ as well as on the encompassing group $G$. It is clear that every inducing subgroup is weakly inducing. A simple weakly inducing subgroup is inducing.

\begin{proposition}\label{prop:locally-inner and inducing}
Let $G$ be a countable group and $N$ a locally inner  subgroup of $G$.  Let $\psi \in \Ch{N}$ be a character. Assume either that
\begin{enumerate}
    \item $N$ is inducing and $\psi$ is non-trivial, or
    \item $N$ is weakly inducing and $\left[N: \ker \psi \right] = \infty$. 
\end{enumerate}
Then the trivial extension of $\psi$ satisfies $\widetilde{\psi} \in \Ch{G}$.
\end{proposition}
\begin{proof}
We prove the proposition under the first assumption, namely that the subgroup $N$ is inducing and the character $\psi$ is non-trivial. We know from Item (3) of Proposition \ref{prop:locally-inner} that   $\widetilde \psi \in \Tr{G}$. 
By Choquet's theorem, there exists a unique  Borel probability measure
$\mu$ on the space $\Ch G$    such that
\[
 \widetilde{\psi}(g)=\int_{\Ch G}\chi(g) \; \mathrm{d} \mu(\chi) \quad \forall g\in G.
\]
Consider the restriction map 
\[
r:\Ch G\to\Ch N, \quad r : \varphi\mapsto\varphi_{\mid N}.
\]
This map is well defined by Item (1)  of Proposition \ref{prop:locally-inner}. It is clearly continuous. Applying   restriction  we get
\[
\psi(g)=\int_{\Ch G}r(\chi)(g) \; \mathrm{d}\mu(\chi)=\int_{\Ch N}\chi'(g) \; \mathrm{d}\, r_{*}\mu (\chi') \quad \forall g\in N.
\]
where $r_{*}\mu$ is the pushforward measure on $\Ch{N}$. 
%This shows that the barycenter of the probability measure  $r_{*}\mu$ on $\Tr{N}$ is $\psi$.
Since $\psi$ is a character of the group $N$, the measure $r_{*}\mu$ must be the Dirac measure supported on $\psi$. This means that $\mu$-almost every character of $G$ restricts to the particular non-trivial character $\psi \in \Ch{N}$. Let $\chi \in \Ch{G}$ be any such character. As the subgroup $N$ is inducing we have $\chi(g) = 0 $ for all $g \in G \setminus N$. Therefore    $ \chi = \widetilde{\psi} $ as required.
We omit the proof in the weakly inducing case, as it follows similarly.
\end{proof}

We conclude \S \ref{sec:locally inner} with a general character classification result for any countable group admitting a locally inner inducing subgroup.

\begin{theorem}
\label{thm:locally inner extensions}
Let $N \lhd G$ be a locally inner inducing subgroup with quotient group $Q = G/N$. 
If the quotient group $Q$ is   ICC then 
 \[
\Ch G\cong\left(\Ch Q\sqcup\Ch{N}\right)/\left(\delta_{e}^{Q}\sim 1_{N}\right).
\]
If the quotient group $Q$ is not ICC then
\[
\Ch G\cong\Ch Q\sqcup \left(\Ch{N}\backslash \{1_{N}\} \right).
\]
These identifications are understood in the sense of topological spaces.
\end{theorem}
\begin{proof}
Let $p : G \to Q$ denote the quotient map. Then the map
$$ f_1 : \Ch{Q} \to \Ch{G}, \quad f_1 :  \varphi  \mapsto \varphi   \circ p $$
is well-defined by Proposition \ref{prop:characters-of-quotients}. Furthermore, it follows from the assumptions together with  Proposition \ref{prop:locally-inner and inducing} that the map
$$ f_2 : \Ch{N} \setminus \{1_N\} \to \Ch{G}, \quad f_2 : \psi \mapsto \widetilde{\psi} $$
is also well-defined. The two maps $f_1$ and $f_2$ are   injective and satisfy $\mathrm{Im}(f_1)\cap \mathrm{Im}(f_2) = \emptyset$. It follows from Item (2) of Proposition \ref{prop:locally-inner} that $\Ch{G} = \mathrm{Im}(f_1)\cup \mathrm{Im}(f_2)$.  A direct verification shows that $f_1$ and $f_2$ are both homeomorphisms onto their respective images. 

To conclude the proof it remains the compute the closures of $\mathrm{Im}(f_1)$ and $\mathrm{Im}(f_2)$. Note that a character $\varphi \in \Ch{G}$ belongs to $\mathrm{Im}(f_1)$ if and only if $\varphi(g) = 1$ for every $g \in N$. Therefore $\mathrm{Im}(f_1)$ is a closed subset of $\Ch{G}$. 
Likewise, any  trace $\varphi \in \Tr{G}$ belonging to the closure $\overline{\mathrm{Im}(f_2)}$ satisfies $\varphi(g) = 0$ for all $g \in G \setminus N$. This means that any character $\varphi \in \mathrm{Im}(f_1) \cap \overline{\mathrm{Im}(f_2)}$ must coincide with the characteristic function of the subgroup $N$. This function is a character of the group $G$ if and only if the quotient $Q$ is ICC \cite[Proposition 7.A.1]{BdlH}. 
\end{proof}

\section{Characters of finitary alternating groups and alternating enrichments} 
\label{sec:alternating}

The topic of the current section is the character theory of finite as well as  finitary alternating groups. We also consider alternating enrichments.
Recall that in this work characters (in general, and of finite groups in particular) are always \emph{irreducible} and \emph{normalized} so that $\varphi(e) = 1$.

\subsection*{Finite alternating groups}
For each $n \in \mathbb{N}$ let $\Alt(\left[n\right])$ and $\Sym(\left[n\right])$ respectively denote the alternating and the symmetric permutation groups on the finite set $\left[n\right]$. For each permutation $\sigma \in \Sym(\left[n\right])$ we consider its \emph{support}, namely
\begin{equation}
\label{eq:support of a permutation}
\mathrm{supp}(\sigma) = \{ i \in \left[n\right] \: : \: \sigma(i) \neq i \}.
\end{equation}

We begin with the following asymptotic result  derived from the character bounds of   Larsen and Shalev \cite{LS}.

\begin{lemma} \label{lem:A_n_char_bound}
Let $\sigma_n \in \mathrm{Alt}(\left[n\right])$ be a sequence of even permutations satisfying
$$\lim_{n\to\infty} \frac{\log(|\supp(\sigma_n)|)}{\log(n)} = 1.$$
Then $\lim_n \varphi_n(\sigma_n) = 0$ for all sequences of non-trivial   characters $\varphi_n \in \Ch{\Alt\left(\left[n\right]\right)}$.
\end{lemma}

Prior to proving Lemma \ref{lem:A_n_char_bound}, we  recall the relationship between characters of the symmetric group $\mathrm{Sym}(\left[n\right])$ and the alternating group $\mathrm{Alt}(\left[n\right])$. 
If $\chi$ is a  character of $\mathrm{Sym}(\left[n\right])$ then either its restriction $\chi |_{\mathrm{Alt}(\left[n\right])}$ is  
a character, or there exist two distinct  characters $\chi_{+}, \chi_- \in \Ch{\mathrm{Alt}(\left[n\right])}$ such that $\chi = \frac{1}{2}\chi^+ + \frac{1}{2}\chi^-$. 
It is known that all characters  of the alternating group $\mathrm{Alt}(\left[n\right])$ are obtained in one of these two ways \cite[Theorem 2.5.7]{JK}.

%The latter case happens if and only if the Young diagram corresponding to $\chi$ is self-conjugate.
%In the latter case $\chi_+^\tau = \chi_-$ for any transposition $\tau \in \mathrm{Sym}(\left[n\right])$. In particular $\chi_+(g) = \chi_-(g)$ for any element $g \in \mathrm{Alt}(\left[n\right])$ with $|\mathrm{\supp(g)}| \le n-2$. 

\begin{proof}[Proof of Lemma \ref{lem:A_n_char_bound}]
By a Theorem of Larsen and Shalev \cite[Theorem 1.3]{LS} there exist $N \in \mathbb N$ such that for all $ n > N$ every  normalized character $\chi \in \Ch{\mathrm{Sym}(\left[n\right])}$ and every element $\sigma \in \mathrm{Sym}(\left[n\right])$ having at most $n^\frac{1}{2}$ fixed points satisfy
$$\vert \chi(\sigma) \vert \leq \dim(\chi)^{-1/5}.$$
Here $\dim(\chi)$ is the dimension of the corresponding irreducible representation. It is well known  that except for the trivial and sign representations, we have
$\dim(\chi) \geq n-1$. In that  case the above bound reads
$$\vert \chi(\sigma) \vert \leq (n-1)^{-1/5 }.$$

It remains to transfer this bound to  the alternating group $\mathrm{Alt}(\left[n\right])$. 
Take some $n > N$. Let  $\varphi \in \Ch{\mathrm{Alt}(\left[n\right])}$ be a non-trivial character and $\sigma \in \mathrm{Alt}(\left[n\right])$ be a permutation having  at most $n^\frac{1}{2}$ fixed points. 
By the discussion in the paragraph preceding this lemma, either the character $\varphi$ is the restriction of some character $\chi \in \Ch{\mathrm{Sym}(\left[n\right])}$ or $\varphi$ is one of $\chi_+, \chi_-$ associated to some character $\chi\in\Ch{\mathrm{Sym}(\left[n\right])}$. In the former case, the desired conclusion follows from the above-mentioned bound of Larsen and Shalev for the symmetric group. In the latter case, the Young diagram associated to the character $\chi$ must be self-conjugate. As such, its first row cannot contain more than $\frac{n}{2}$ cells. In that situation $|\varphi(\sigma)| \le q^{n^{\frac{1}{2}}}$ for some constant $0 < q < 1$, see \cite[Theorem 5.4]{Roichman}.
\end{proof}

%   \ad{Apply Doron Puder and Sage math!}
%   \iv{One possible way of doing this is as follows. Assume $n$ is equal to $1$ mod 4, so that the element $t_n=(135...n)$ is in $\mathrm{Alt}(\left[n\right])$. Then $y_n=[\sigma_n,t_n]$ is a full cycle. To see that $\varphi_n(y_n)\to 0$, assume first we are dealing with the symmetric group and that $\varphi_n$ are faithful (i.e non-trivial and not the sign). Unless $\varphi_n$ is the character whose corresponding Yang tableaux is the symmetric 'L'-shaped tableaux, then $\varphi_n(y_n)=0$ for all $n$. If it is the symmetric 'L'-shaped tableaux then $\varphi_n(y_n)$ is $1$ over the dimension of the corresponding representation. This dimension goes to infinity as $n$ goes to infinity, and so we are done.
%   
%   Now, the characters of the alternating group split to two types. There are the symmetric ones- these are just restrictions of characters of the symmetric group whose corresponding Yang tableaux is non-symmetric when reflecting along the diagonal axis. Thus, the previous paragraph shows that everything works with symmetric characters of $\mathrm{Alt}(\left[n\right])$. The non-symmeric characters come in pair of $\varphi$ and $\overline{\varphi}$. They arise by restricting a character of the symmetric group with a symmetric Yang tableaux to the alternating group, and decomposing it to these two irreducible characters. Is $y_n$ also a suitable choice for them?}

\subsection*{Finitary alternating groups}
Let $Y$ be an infinite countable set. We let  $\Symfin{Y}$ denote the group  of all finitely supported permutations of the set $Y$. Likewise, the group $\Altfin{Y}$ consists of all even such permutations. 

\begin{proposition}
\label{prop:alt infinity is ranked}
Fix a non-trivial character $\varphi \in \Ch{\Altfin{Y}}$. Then
$$\lim_{n \to \infty} \sup \{ |\varphi(g)| \: : \: g \in \Altfin{Y}, |\mathrm{supp}(g)| = n \} = 0.$$
\end{proposition}
\begin{proof} 
Consider an arbitrary sequence of elements $g_n \in \Altfin{Y}$ such that the sequence $a_n = |\mathrm{supp}(g_n)|$ is strictly monotone increasing. To establish the lemma, it will suffice to show that $\lim \varphi(g_n) = 0$.  Up to conjugating the elements $g_n$ appropriately, we may find a nested sequence of subsets $Z_n \subset Y$ with $|Z_n| = a_n$ such that $\mathrm{supp}(g_n) = Z_n$.  In other words, each element $g_n$ belongs to the finite subgroup  $\Alt(Z_n) $ of $\Altfin{Y}$. Up to a further conjugation, we may assume  that $\bigcup Z_n = Y$.
 Write the restriction $\psi_n \in \Tr{\Alt(Z_n)}$ of the character $\psi$ to the subgroup $\Alt(Z_n)$ as $\psi_n = \alpha_n   + (1-\alpha_n) \psi'_n$,  where $ 0 \le \alpha_n \le 1$ and $\psi'_n$ is a trace which does not dominate\footnote{The notion of dominated traces is defined in \cite[11.C]{BdlH}, see also \cite[\S2]{levit2023characters}.} the trivial trace. Note that certainly $\lim \psi'_n(g_n)  = 0$ by  Lemma \ref{lem:A_n_char_bound} and its proof. To conclude the proof we argue that   $\lim \alpha_n = 0$. Assume towards contradiction (and up to passing to a subsequence) that $\alpha = \lim \alpha_n > 0$. Let $N \in \mathbb{N}$ be sufficiently large so that $\alpha_n \ge \frac{\alpha}{2}$ for all $n > N$. Then the non-trivial character $\varphi$ dominates the trivial character $\frac{\alpha}{2} 1$, which is a contradiction to Proposition 2.2 and Corollary 2.3 of \cite{levit2023characters}.
\end{proof}

Thoma \cite{Thoma-symmetric} has obtained a celebrated character classification for the two groups $\Symfin{Z}$ and $\Altfin{Z}$. It is possible to obtain a proof of Proposition \ref{prop:alt infinity is ranked}   directly from that classification. We refer the reader also to \cite{thomas2022characters}.

\subsection*{Permutations with infinite support}
%We turn to discuss automorphisms of $\AltZ$. Note
%that $\AltZ$ is a normal subgroup of the group %$\Sym(\Z)$ consisting
%of all permutations on $\Z$. Thus the group $\Sym(\Z)$ acts via conjugation
%on $\AltZ$, and this action is locally-inner (see Example \ref{exam:Sym_Alt_loc_inner}). 
% This defines a homomorphism:
% \[
% \Sym(\Z)\to\Aut{\AltZ}
% \]
% It is easy to see that this map is injective.  \iv{It is in fact known
% to be an isomorphism. We don't actually need this, but is makes things more proper. We should either add a ref, write a proof, or avoid it all together. For now let's have it}
% \ad{can you find a reference? I think this is nice remark}. It is easy to see that
% an element $\sigma\in\Sym (\Z)$ admits an infinite orbit in its action
% on $\Z$ if and only the corresponding element of $\Aut{\AltZ}$ admits
% an infinite orbit in its action on $\AltZ$.

Let $Y$ be an infinite countable set.

\begin{lemma}
    \label{lem:commutators of permutations}
Let $r\in \Sym(Y)$ be an infinitely supported permutation.
Then there exists a sequence of finitely supported permutations $x_{k}\in \Altfin{Y}$ such that  for every fixed $m\in \mathbb{N}$ 
\[
\label{eq:supports go to infinity for permutations}
    \lim_{k \to \infty} |\mathrm{supp}(\left[r,x_k\right]^{-1} \left[r,x_m\right])| = \infty.
\]
\end{lemma}
\begin{proof}
The proof will depend on whether the permutations $r$ admits an infinite orbit or not. 

Assume first that $r$ does admit an infinite orbit. Restricting attention to this orbit, we may and will assume without loss of generality that  $Y=\Z$ and $r(i) = i+1$ for all $ i \in \Z$. Denote    $s_i = (i,i+1)$ so that each $s_i$ is a   transposition.   For each \emph{odd} value $k \in \N$ consider the permutation $x_{k} \in \AltZ$ given by
\[
x_{k}=\prod_{i=0}^{k}s_{2(i+k^{2})}=s_{2k^{2}}\cdot s_{2+2k^{2}}\cdots s_{2k+2k^{2}}.
\]
It is easy to verify directly that the conjugation of $x_k$ by $r$ is the permutation
\[
x_k^r =\prod_{i=0}^{k}s_{1+2i+2k^{2}}=s_{1+2k^{2}}\cdot s_{3+2k^{2}}\cdots s_{1+2k+2k^{2}}.
\]
It follows that the commutator element $y_{k}=[r,x_k]=x_k^{-1} x_k^r$ has support  $$\supp (y_k) = \left\{ j \
\: : \: 2k^{2} \le j \le  2k^{2}+2k+2 \right\}. $$
The size of these sets is unbounded and they are pairwise disjoint. The desired conclusion follows. 

We pause to make a useful observation. Consider a permutation $z\in \Sym(Y)$ and suppose that it has admits at least two orbits  $S_1$ and $S_2$.  Let $t =(x_1,x_2)$ be an arbitrary transposition for some pair of points $x_1 \in S_1$ and $x_2 \in S_2$.  Then the commutator $[z,t]$ is the product of two transpositions supported inside $S_1\cup S_2$.  
More generally, suppose that $z$ admits $2k$ orbits    $S_1,S_2,...,S_{2k}$ for some $k\in \N$.  Choose arbitrary points $x_i \in S_i$. For each $i=1,...,k$ consider the transposition $t_i = (x_{2i-1},x_{2i})$ and write   $t=t_1\cdots t_k$. Then the commutator $[z,t]$ splits as a product  $[z,t]=[z,t_1] \cdots [z,t_k]$. Thus $[z,t]$  is the product of $2k$ transpositions supported in $S_1 \cup \cdots \cup S_{2k}$.
This observation will help us deal with the following situation.

Assume next that the permutation $r$ is infinitely supported but has no infinite orbits. As such $r$ has infinitely many finite orbits. Enumerate and denote these orbits   by $S_i$ where $i \in \mathbb{N}$. Choose arbitrary points $x_i \in S_i$ for each $i \in \mathbb{N}$. Consider the transposition $t_i = (x_{2i-1},x_{2i})$ for each $i \in \mathbb{N}$.
We set
\[
    x_{k}=\prod_{i=1}^{k}t_{i+k^{2}}= t_{1+k^{2}}\cdots t_{k+k^{2}}.
\]
%Thus, $x_1= t_3$, $x_2= t_5\cdot t_6$, $x_3=t_{10}\cdot t_{11} \cdot t_{12}$, etc. 
Using the observation above we see that the commutator $[r,x_k]$ is the product of $2k$-many transpositions whose supports are contained 
in $\bigcup_{i=1}^{2k} S_{i+2k^2}$. In particular, the supports of $[r,x_k]$ are pairwise disjoint and of unbounded size, as desired. 
\end{proof}

\begin{proposition}
\label{prop:AltZ is inducing}
Let $H$ be a countable group acting on an infinite countable set $Y$. 
Consider the semidirect product $G = H \ltimes \Altfin{Y}$ corresponding to the resulting action of $H$ on $\Altfin{Y}$ by automorphisms. 
If every non-trivial element of $H$ has infinite support in its action   on $Y$ then   the normal subgroup $\Altfin{Y}$ of $G$ is inducing.
\end{proposition}

\begin{proof}
Fix an element $g \in G$ and write $g = hx$ for some pair of elements $h \in H$ and $x \in \Altfin{Y}$. Assume that $g \notin \Altfin{Y}$, or what is equivalent  $h$ is non-trivial.  As the action of the element  $h$ on the set $Y$ admits an infinite
orbit, whereas the permutation $x$ is finitely supported, their product $g$ admits an infinite
orbit in its action on $Y$ as well. By  Lemma \ref{lem:commutators of permutations} there exists a sequence
of elements $x_{k}\in \Altfin{Y}$ such that the commutators $y_n = \left[g,x_n\right] \in \Altfin{Y}$ satisfy  $\lim_n |\mathrm{supp}(y_n^{-1} y_m )| = \infty$ for each fixed $m \in \mathbb{N}$.

Note that $\Altfin{Y}$ is a locally inner subgroup of $G$, see Example \ref{exam:Sym_Alt_loc_inner} and the remarks following it. Let $\varphi \in \Ch{G}$ be any character such that the restriction of $\varphi$ to the normal subgroup $\Altfin{Y}$ is not-trivial. We know that this restriction is a character by Proposition \ref{prop:locally-inner}. It follows from Proposition \ref{prop:alt infinity is ranked} that 
$
\lim_{n\to\infty}\varphi(y_{n}^{-1}y_{m})=0 
$
for each fixed $m \in \mathbb{N}$.   The desired conclusion follows from  our vanishing result for  asymptotically orthonormal sequences, see  Lemma \ref{lem:gen_bekka}.
\end{proof}

It is worth noting that every automorphism of the group $\Altfin{Y}$ arises from conjugation by some element of $\Sym(Y)$ \cite[Theorem 8.2A]{dixon1996permutation}.

\subsection*{Characters of alternating enrichments} Recall that the construction of alternating enrichments was introduced in \S\ref{sec:diagonal products from alternating and elementary enrichments}.

\begin{corollary} 
\label{cor:chars of alternating enrichments}
Let $G$ be an infinite group with  alternating enrichment $\mathscr A(G)$.  If the  group $G$ is   ICC then 
 \[
\Ch {\mathscr A(G)} \cong\left(\Ch G\sqcup\Ch{\Altfin{G}}\right)/\left(\delta_{e}^{G}\sim 1_{\Altfin{G}}\right).
\]
If the  group $G$ is not ICC then
\[
\Ch {\mathscr A(G)} \cong\Ch G\sqcup \left(\Ch{\Altfin{G}}\backslash \{1_{\Altfin{G}}\} \right).
\]
\end{corollary}
\begin{proof}
This corollary follows immediately from   Theorem \ref{thm:locally inner extensions} by taking into account the information that the normal subgroup $\Altfin{G}$ is locally inner (see  Example \ref{exam:Sym_Alt_loc_inner}) as well as inducing (by   Proposition \ref{prop:AltZ is inducing}). 
\end{proof}

%\ad{redundant corollary} why is it reduntant? it is a nice aplication, and specifically we need it

\begin{corollary}
\label{cor:infinite orbit}
Consider some semidirect product $G = \Z \ltimes \AltZ$. If the automorphism of the group $\AltZ$ corresponding to the generator of the group $\Z$ arises from a permutation  with infinite orbit of the underlying set of $\Z$ then
\[
\Ch G=S^{1}\sqcup\left(\Ch{\AltZ}\backslash\{1_{\AltZ}\}\right).
\]
\end{corollary}
\begin{proof}
The abelian group $\Z$ is certainly not ICC. Its character space coincides with its Pontryagin dual, namely   $\Ch{\Z} = \widehat{\Z} \cong S^1$.
\end{proof}

%\ad{This remark is not clear enough} - what is not clear? i
The above two statements deal  with the finitary alternating group $\AltZ$. Alternating groups are better suited for our purposes. In any case, these two statements hold true verbatim (and with the same proofs) if the group $\AltZ$ is replaced by the finitary symmetric group $\SymZ$. 

\section{Characters of finitary special linear groups and elementary enrichments}
\label{sec:elementary}

Let $p$ be a prime. Fix the finite field 
$F = \mathbb{F}_{q}$ for some $q =p^k$ and $k \in \mathbb{N}$. Let $\overline{F}$ denote the algebraic closure of the field $F$. We consider the character theory of the finite as well as finitary special linear groups over the field $F$. 
Furthermore, we study characters of elementary enrichments. We remind the reader that for us,  characters   are always extremal as a point in $\mathrm{Tr}(G)$ and normalized so that $\varphi(e) = 1$.

\subsection*{Finite special linear groups}

For each $n \in \mathbb{N}$ we consider the finite special linear group $\mathrm{SL}_n(F)$.

\begin{definition}
\label{def:support of a matrix}
The support $\mathrm{supp}(g)$ of an element $g \in \mathrm{SL}_n(F)$  is the codimension of the largest eigenspace in its canonical action on the vector space $\overline{F}^n \cong F^n \otimes \overline{F}$.
\end{definition}

Note that the support of a central element  is zero. The support of a permutation matrix in the sense of Definition \ref{def:support of a matrix} is related to the support of the underlying permutation.

\begin{lemma}
\label{lemma:support of a permutation matrix}
Let $g \in \mathrm{SL}_n(F)$ be a permutation matrix associated to the even permutation $\sigma \in \Alt(n)$. If  the permutation $\sigma$ has an orbit of size $d$ then  $\mathrm{supp}(g) \ge d-1$. 
\end{lemma}
\begin{proof}
Assume that $\sigma$ has an orbit of size $d$. Let $V \le \overline{F}^n$ be the $d$-dimensional subspace corresponding to this orbit.  The restriction of the matrix $g$ to the invariant subspace $V$ is diagonizable with $d$-many distinct eigenvalues $1,\omega,\ldots,\omega^{d-1}$. Here $\omega \in \overline{F}$ is a $d$-th root of unity. This certainly implies that $\mathrm{supp}(g) \ge d-1$.
\end{proof}

The size of the support (in the sense of Definition   \ref{def:support of a matrix}) can be used to upper-bound values of normalized characters of finite special linear groups.

\begin{theorem}[Theorem 1.2.1 of \cite{larsen2011waring}]
%\marginpar{how does this work if $g$ is a central element?}
\label{thm:character bounds for SLn}
Let $g \in \mathrm{SL}_n(F)$ be any non-trivial element and $\varphi \in \Ch{\mathrm{SL}_n(F)}$ be any non-trivial normalized character. Then
$$ |\varphi(g)| \le q^{ -\sqrt{\mathrm{supp}(g)}/481}.$$
\end{theorem}

Unlike the case of finite alternating groups, normalized characters of finite special linear groups admit a universal upper bound.

\begin{theorem}[Theorem 2.2 of \cite{gluck1997characters}]
\label{thm:universal upper bound}
Assume $q > 10$. There is some constant $c < 1$ such that $|\varphi(g)| < c$ for all $n \in \mathbb{N}$, all non-trivial characters $\varphi \in \Ch{\mathrm{SL}_n(F)}$ and all non-central elements $g \in \mathrm{SL}_n(F)$.
\end{theorem}

\subsection*{Finitary special linear groups}

Let $\SLfin{F}$ denote the direct limit of the special linear groups $\mathrm{SL}_n(F)$ as $n$ tends to infinity. 
Alternatively, we may view the group $\SLfin{F}$ as consisting of finitary matrices  over a countably infinite index set. This means that the entries $g_{i,j}$  of each matrix $g \in \SLfin{F}$ differ from $\delta_{i,j}$ only for finitely many pairs of indices $i$ and $j$.

%We extend the definition of the support to finitary matrices.
% \begin{definition}
% \label{def:support of a finitary matrix}
% The support $\mathrm{supp}(g)$ of an element $g \in \SLfin{F}$  is the codimension of the  eigenspace with eigenvalue $1$ in its canonical action on the infinite-dimensional vector space $\bigoplus_\mathbb{N}  \overline{F} = \overline{F} \otimes \bigoplus_\mathbb{N} F$.
% \end{definition}

% In other words, for an element $g \in \SLfin{F}$, its support $\mathrm{supp}(g)$ is the same as the rank of the matrix $g-I$. 
 
\begin{proposition}
\label{prop:infinite permutations are ranked}
Let $\varphi \in \Ch{\SLfin{F}}$ be any non-trivial character. Then
$$\lim_n \sup \{ |\varphi(g) | \: : \: g \in \SLfin{F}, \; \mathrm{rank}(g - \mathrm{I}) = n \} = 0.$$ 
\end{proposition}

\begin{proof} 
Let $g_n \in \SLfin{F}$ be any sequence of elements with $a_n = \mathrm{rank}(g_n - \mathrm{I})$ strictly monotone increasing. To establish the proposition it will suffice to show that $\lim \varphi(g_n) = 0$.

Up to conjugating each element by a suitable permutation matrix, we may and will assume that  $g_n \in G_n$ where the subgroup $G_n = \mathrm{SL}_{2a_n }(F)  $ is identified   with an \enquote{upper-left corner} of the group $\SLfin{F}$. In this way $1 \in \overline{F}$ will certainly be the largest eigenvalue of $g_n$ regarded as an element of $G_n$ acting of the finite-dimensional subspace $\overline{F}^{2a_n}$. Therefore $\mathrm{supp}(g_n) = a_n$, where the notion of support is  as in   Definition \ref{def:support of a matrix} and $g_n$ is regarded as an element of the finite group $G_n$.
%\ref{def:support of a finitary matrix} coincide for  $g_n$, regardless of whether it is  considered as an element of $\SLfin{F}$ or of $G_n$.

Write the restriction $\psi_n$ of the character $\varphi$ to the subgroup $G_n$ as a convex combination $\psi_n = \alpha_n \cdot 1  + (1-\alpha_n) \psi'_n$ where $\psi'_n$ is a trace which does not dominate the trivial character $1 \in \Ch{G_n}$. We know that  $\lim  \psi'_n(g) = 0$ by Theorem \ref{thm:character bounds for SLn}. To conclude to proof it remains to show that $\lim  \alpha_n = 0$.
Assume towards contradiction (and up to passing to a subsequence) that $\alpha = \lim \alpha_n > 0$. Let $N \in \mathbb{N}$ be sufficiently large so that $\alpha_n \ge \frac{\alpha}{2}$ for all $n > N$. Then the non-trivial character $\varphi$ dominates the trivial character $\frac{\alpha}{2} \cdot 1$, which is a contradiction to Proposition 2.2 and Corollary 2.3 of \cite{levit2023characters}.
\end{proof}

Alternatively,  Proposition \ref{prop:infinite permutations are ranked} can be deduced quite directly from the character classification of the group $\SLfin{F}$ available in \cite{skudlarek1976unzerlegbaren}. Proposition \ref{prop:infinite permutations are ranked} is the special linear groups analogue of Proposition \ref{prop:alt infinity is ranked}.

\begin{lemma}
\label{lemma:support of finitary permutation matrix}
Let $g \in \SLfin{F}$ be a permutation matrix associated to the even permutation $\sigma \in \AltZ$. Then 
$$\mathrm{rank}(g-\mathrm{I}) = |\supp(\sigma)| - l \ge \frac{1}{2}|\supp(\sigma)| $$ 
where $l$ is the number of orbits of the permutation $\sigma$.
\end{lemma}

Lemma \ref{lemma:support of finitary permutation matrix} is very  similar to Lemma \ref{lemma:support of a permutation matrix}.

\begin{proof}[Proof of Lemma \ref{lemma:support of finitary permutation matrix}]
For each orbit of size $d$ of the permutation $\sigma$, the matrix $g$ has a $d$-dimensional invariant subspace with distinct eigenvalues $1,\omega,\ldots,\omega^{d-1}$ where $\omega$ is a $d$-th root of unity. In particular, all but one of these eigenvalues are distinct from $1 \in \overline{F}$. The desired conclusion follows.
\end{proof}
 
\begin{proposition}
\label{prop:commutators and ranks in SLinfty}
Let $\sigma$ be the automorphism of the finitary special linear group $\SLfin{F}$ corresponding to conjugation by a permutation matrix admitting an  infinite orbit. Then there is a sequence of elements $x_n \in \SLfin{F}$ such that 
$$ \lim_n \mathrm{rank}(\left[\sigma,x_n\right]^{-1} \left[\sigma,x_m\right] - \mathrm{I}) = \infty$$
for each fixed $m \in \mathbb{N}$.
\end{proposition}
\begin{proof}
The desired conclusion follows by letting the $x_n$'s be the permutation matrices associated to the particular permutations constructed in Lemma \ref{lem:commutators of permutations}.  Lemma \ref{lemma:support of finitary permutation matrix} can then be used to estimate the rank of the resulting commutators and their products.
\end{proof}

\subsection*{Characters of elementary enrichments}

The elementary enrichment construction was introduced in \S\ref{sec:diagonal products from alternating and elementary enrichments}. We consider the character theory of these enrichments.

\begin{corollary} 
\label{cor:chars of elementary enrichments}
Let $G$ be an infinite group with elementary enrichment ${\mathscr E(G)}$. If the   group $G$ is   ICC then 
 \[
\Ch {\mathscr E(G)} \cong\left(\Ch G\sqcup\Ch{\SLfin{V_G}}\right)/\left(\delta_{e}^{G}\sim 1_{\SLfin{V_G}}\right).
\]
If the   group $G$ is not ICC then
\[
\Ch {\mathscr E(G)} \cong\Ch G\sqcup \left(\Ch{\SLfin{V_G}}\backslash \{1_{\SLfin{V_G}}\} \right).
\]
\end{corollary}
\begin{proof}
Note that the normal subgroup $\SLfin{V_G}$ of the elementary enrichment $\mathscr E (G) = G \ltimes \SLfin{V_G}$ is locally inner.  It is inducing, as can be seen by combining   Proposition \ref{prop:commutators and ranks in SLinfty} with our vanishing result Lemma \ref{lem:gen_bekka}. The  desired result now follows as a special case of Theorem    \ref{thm:locally inner extensions}.
\end{proof}

\section{Characters of diagonal products}
\label{sec:diagonal products}

Let  $G$ be a  group generated by a finite set $S$ and  equipped with a LEF approximation $\theta_n :G \to G_n$ (see Definition \ref{def:partial homo and LEF approx}). In addition, we assume that the LEF approximation $\theta_n$ is essential and by simple groups (see Definition \ref{def:essential LEF}). 

This section is dedicated to analyzing the character theory of the diagonal product $\bigotimes_n G_n$. In particular, we  will prove  Theorem \ref{thm:intro compatbile characters}. Throughout the current section, we will be using notations and terminology set up in \S\ref{sec:LEF groups and diagonal products}.

 The following general observation will be useful. It follows from the local nature of the definition of traces (we leave the straightforward proof to the reader).

\begin{lemma}
\label{lemma:accumulation of partially defined traces is a trace}
Let $f_n : G \to H_n$ be a partial homomorphism and 
  $\varphi_n \in \Tr{H_n}$ be a   sequence of traces. Then any accumulation point in the space $\ell^\infty(G)$ of the sequence $\varphi_n \circ f_n$ belongs to  $ \Tr{G}$.
\end{lemma} 
%\ad{This lemma is proved somewhere in \cite{FFGS}, should just cite}\iv{I agree, I erased it. I couldn't find the right place to refer to (Alon, do you know where?) but even if we don't have a ref, i would "leave this as an exercise"}

\subsection*{Diagonal products and compatible traces}

Let $\Gamma = \bigotimes G_n$ be the diagonal product of the finite groups $G_n$ associated to the LEF approximation  $\theta_n : G \to G_n$ and to the generating set $S$. Let $t : \Gamma \to G$ be the associated tail map. Denote $U = \ker t$ so that $ U = \bigoplus G_n$ by Lemma \ref{lem:sum_contained_in_diag_prod}. In addition, recall that the normal subgroup $U$ is  locally inner  in $\Gamma$ by  Lemma \ref{lem:locally inner for diagonal products}.   

Let $\zeta \in \Ch{\Gamma}$ be any character. The fact that the  subgroup $U$ is locally inner implies that the restriction $ \varphi = \zeta_{|U}$ is a character of  $U$, see Proposition \ref{prop:locally-inner}. Recall that the character space of the direct sum $U$ is given by $\Ch{U} \cong \prod_i \Ch{G_i}$, see Proposition \ref{prop:char_direct_prod}.
As such, the character $\varphi$ can be written as $\varphi = \otimes \varphi_i $ for some uniquely determined sequence  of characters $\varphi_i \in \Ch{G_i}$.

\begin{definition}
\label{def:compatible trace}
An element $\tau \in G$ is \emph{null} for a character $\varphi \in \Ch{U}$ if   \emph{every} element $g = (g_i) \in \Gamma$ with $t(g) = \tau$ satisfies 
\begin{equation}
\label{eq:definition of null}
 \prod_{i=1}^\infty | \varphi_i(g_i) | = 0.
\end{equation}
A trace $\psi \in \Tr{G}$ is \emph{compatible} with a character $\varphi \in \Ch{U}$ if $\psi(\tau)=0$  for every element $\tau \in G$ which is null for $\varphi$.
\end{definition}

Every trace $\varphi$ satisfies $\|\varphi\|_\infty \le 1$, so that the partial products  appearing in Equation (\ref{eq:definition of null}) are monotone non-increasing, hence the infinite product of absolute values always exists.
Note that the subset of $\Tr{G}$ consisting of those traces compatible with each given character $\varphi \in \Ch{U}$ is  closed.

\begin{theorem} 
\label{thm:traces over a character in a tail group}
Fix a character $\varphi \in \Ch{U}$ given by $\varphi 
 = \otimes \varphi_i$ with $\varphi_i \in \Ch{G_i}$. Then
$$\{\zeta \in \Tr{\Gamma} \: : \: \zeta_{|U} = \varphi \}  \cong \{ \psi \in  \Tr{G} \: : \: \text{$\psi$ is compatible with $\varphi$} \}.$$
The formula for this correspondence is $\zeta = \widehat{\varphi} \cdot (\psi \circ t)$ where $\widehat{\varphi} \in \Tr{\Gamma}$ is a certain auxiliary trace.
\end{theorem}
\begin{proof}
We will construct a pair of continuous affine maps $F_1$ and $F_2$ in both directions. These two maps will depend on a choice of an auxiliary object, as follows. 

Recall that the pointwise product of two traces is a trace \cite[Proposition C.1.6]{bekka2014kazhdan}. In particular, the function $\widehat{\varphi}_n : g \mapsto  \prod_{i=1}^n \varphi_i(g_i)$ is a trace on the group $\Gamma$ for each $n \in \mathbb{N}$. This trace depends only on the projection of each element to the subgroup  $U_n$. Fix an arbitrary accumulation point  $\widehat{\varphi}$ of the sequence $\widehat{\varphi}_n$ in the compact space of traces $\Tr{\Gamma}$. Let $n_k$ be some subsequence  such that $\widehat{\varphi}_{n_k} \to \widehat{\varphi}$.

% the sentence you added already exists below
%\ad{added this sentence:} Note that if $\tau\in G$ is not null for $\varphi$, then every partial limit of $\widehat{\varphi}_n(g)$ for $g \in \Gamma$ with $t(g) = \tau$ is non-zero, as $| \widehat{\varphi}_n(g) |$  is a nonincreasing sequence, as such, $\widehat{\varphi}(g) \neq 0$.

We construct the map $F_1$ in one direction. Let $\zeta \in \Tr{\Gamma}$ be a trace with $\zeta_{|U} = \varphi$.
Write $\Gamma = U_n \times \Gamma_{n+1}$ for each $n$. The restriction of $\zeta$ to the subgroup $U_n$ is a character. Take any element $g \in \Gamma$ and write $g = g_1\cdots g_n g'$ where $g' \in \Gamma_{n+1} \le \Gamma$ and $g_i \in G_i$ for each $ 1 \le i \le n$.  Corollary \ref{cor:product restricting to char on one} implies that 
\begin{equation}
\label{eq:reformulate 1}
\zeta(g) = \varphi_1(g_1) \cdots \varphi_n(g_n) \zeta (g').
\end{equation}
Having fixed the element $g \in \Gamma$ and assuming that the index $n$ is sufficiently large so that $g' = \Phi_{n+1} \circ t(g)$ we can reformulate Equation (\ref{eq:reformulate 1}) as
\begin{equation}
\label{eq:reformulate 2}
\zeta(g) = \varphi_1(g_1) \cdots \varphi_n(g_n) \zeta (\Phi_{n+1} \circ t(g)).
\end{equation}
For each $n$, set $\psi_n : G \to \mathbb{C}$ by $\psi_n = \zeta \circ \Phi_{n+1} $.
Equation (\ref{eq:reformulate 2}) can be rewritten as
\begin{equation}
\label{eq:reformulate 3}
\zeta = \varphi_1 \cdots \varphi_n \cdot ( \psi_{n} \circ t).
\end{equation}
Equation (\ref{eq:reformulate 3}) holds true for every element $g \in \Gamma$ and for all $n$ sufficiently large (depending on $g$).

Note that each element $g \in \Gamma_{n}$ can we decomposed uniquely as  $g = h g'$ with $h \in G_n$ and $g' \in \Gamma_{n+1}$. According to Corollary \ref{cor:product restricting to char on one}, every such element $g$ satisfies 
\begin{equation}
\label{eq:wierd splitting}
\psi_{n-1}(t(g)) = \varphi_n(h) \cdot \psi_{n}(t(g')).
\end{equation} 

We claim that the limit $\lim_k \psi_{n_k}$ exists in the space $\ell^\infty(G)$ (pointwise convergence). For a given element $\tau \in G$, there are two cases to consider:
\begin{itemize}
    \item The element $\tau \in G$ is null for $\varphi$. Then by taking the limit $n\to\infty$ in     Equation (\ref{eq:reformulate 3}), we see that $\zeta(g) = 0 $ for all elements $g \in \Gamma$ with $t(g) = \tau$. There are two subcases:
    \begin{itemize}
        \item Every element $g \in \Gamma$ with $t(g) = \tau$ admits  infinitely many $n$'s where  $\varphi_n(g_n) = 0$. In this case, repeated application of Equation (\ref{eq:wierd splitting}) shows that $\psi_n(\tau) = 0$ for all $n$ sufficiently large.
        \item There is some element $g = (g_n) \in \Gamma$ with $t(g) = \tau$ such that $\varphi_n(g_n) \neq 0$ for all $n$ sufficiently large. In this case, we may change finitely many coordinates of the element $g$ to be the identity and assume without loss of generality that $\varphi_n(g_n) \neq 0$ for all $n$. It now follows from Equation (\ref{eq:reformulate 3}) that $\psi_n(\tau) = 0$ for all $n$ sufficiently large.
    \end{itemize}
    In either situation, we conclude that $\psi_n(\tau) = 0$ for all $n$ sufficiently large. In particular $  \lim_n \psi_n(\tau) = 0$   (and this limit exists even without passing to a subsequence).
    \item The element $\tau \in G$ is not null for $\varphi$. By definition, there is some element $g = (g_i) \in \Gamma$ with $t(g) = \tau$ so that
    $$ \lim_{n \to \infty} \prod_ {i=1}^n|\varphi_i(g_i)| > 0.$$
Note that the limit $\widehat{\varphi}_{n_k}(g)$ exists along the subsequence $n_k$ and is different from $0$.  Therefore the limit as $k\to\infty$ of the expressions
\begin{equation}
\label{eq:for continuity}
\psi_{n_k}(\tau) = \frac{\zeta(g)}{ \widehat{\varphi}_{n_k}(g)}  \end{equation}
exists as well.   Observe that the value of the limit is  independent of the particular choice of the element $g$, since the left-hand side $\psi_{n_k}(\tau)$ depends only on  $\tau$.
\end{itemize}

Denote  $\psi = \lim_k \psi_{n_k} \in \ell^\infty(G)$.  We now claim that  $\psi \in \Tr{G}$. This follows from the fact that $\Phi_n$ are partial homomorphisms, see Lemma \ref{lemma:accumulation of partially defined traces is a trace}.
 This defines an affine map $F_1 : \zeta \mapsto \psi$. The continuity of the map $F_1$ follows from Equation (\ref{eq:for continuity}). Further, the trace $\psi$ is compatible with the character $\varphi$, as follows from the discussion in the first bullet above.

We now construct the map $F_2$ in the converse direction. Let $\psi \in \Tr{G}$ be any trace compatible with the fixed character $\varphi$. We define a trace $\zeta \in \Tr{\Gamma}$ as follows 
$$ \zeta(g) = \widehat{\varphi}(g) \cdot  (\psi \circ t)(g) \quad \forall g \in \Gamma.$$
This is a trace, being a pointwise product of the two traces $\widehat{\varphi}$ and $\psi \circ t$ on $\Gamma$, see Proposition \ref{prop:prod fin dim traces}. Note that $\zeta_{|U} = \varphi$ as required. This defines a second continuous affine map $F_2 : \psi \mapsto \zeta$.

It remains to verify that the composition $F_2 \circ F_1$ is the identity on the set $\{\zeta \in \Tr{\Gamma} \: : \: \zeta_{|U} = \varphi \}$ and that the composition $F_1 \circ F_2$ is the identity on the set $\{\psi  \in \Tr{G} \: :\: \text{$\psi$ is compatible with $\varphi$} \}$. 

\begin{itemize}
\item
For $F_1 \circ F_2$: let $\psi \in \Tr{G}$ be any trace compatible with  the character $\varphi$. Then the trace  $\psi' = (F_1 \circ F_2) (\psi) \in \Tr{G}$ is also compatible with $\varphi$ by construction. Therefore it suffices to check that $\psi(\tau) = \psi'(\tau)$ only for elements $\tau \in G$ that are not null for $\varphi$.
%Alon: what you wrote is not a part of the requirement
%\ad{Why is this enough? need to explain why for non-null elements, indeed $\hat{\varphi}_{n_k}$ to something nonzero}. 
Denote $\zeta = F_2 (\psi) = \widehat{\varphi} \cdot (\psi \circ t)$. Consider a non-null element $\tau \in G$. Take any element $g \in \Gamma$ with $t(g) = \tau$ for which $\ \prod_{i=1}^\infty |\varphi_i(g_i)| > 0$. We have
$$ \psi'(\tau) = \lim_k \psi'_{n_k}(\tau) = \lim_{k\to\infty} \frac{\zeta(g)} {\widehat{\varphi}_{n_k}(g)} = \lim_{k\to\infty} \frac{\widehat{\varphi}(g) \cdot \psi(\tau) } {\widehat{\varphi}_{n_k}(g)} = \psi(\tau).$$

\item
For $F_2 \circ F_1$: let $\zeta \in \Tr{\Gamma}$ be any trace with $\zeta_{|U} = \varphi$. Denote $\psi = F_1 (\zeta)$. Then for any element $g \in \Gamma$ we have
 \begin{align*}(F_2 \circ F_1)(\zeta)(g) =  \widehat{\varphi}(g) \cdot \psi(t(g)) = \lim_{k\to\infty} \widehat{\varphi}_{n_k}(g) \cdot \psi_{n_k} (t(g))  = \lim_{k\to\infty} \zeta(g) = \zeta(g).\end{align*}
\end{itemize}
This concludes the proof.
\end{proof}

\begin{remark}
While the spaces on both sides of the above theorem are well-defined, the isomorphism between them is not canonical as it depends on a choice of a certain accumulation point $\widehat{\varphi}$ (and subsequence $n_k$).   
\end{remark}

\subsection*{Compatible characters}
The set of traces of $G$ which are compatible with a fixed trace $\varphi\in \Tr{U}$ form a compact and convex subset of $\Tr{G}$. We are interested in the extreme points of this space.

\begin{definition}
\label{def:compatible character}
A trace $\psi \in \Tr{G}$  compatible with a given fixed character  $\varphi \in \Ch{U}$ is called a \emph{compatible character} if $\psi$ cannot be written as a non-trivial convex combination of two  traces on the group $G$ both compatible with $\varphi$.
\end{definition}

As such, the proof of Theorem \ref{intro:thm:diag_prod_stab_main} follows immediately from Definition \ref{def:compatible character}, and the affine correspondence established in Theorem \ref{thm:traces over a character in a tail group}.
 We spell out Definition \ref{def:compatible character} in  two extreme cases:
\begin{itemize}
    \item If every non-trivial element $g \in G$ is null for the character $\varphi$ then the only compatible trace (which is also a compatible character) is $\delta_e \in \Tr{G}$.
    \item If no element $g \in G$ is null for $\varphi$ then all traces are compatible, and compatible characters are just characters of $G$. 
\end{itemize}
  
\begin{example}
\label{example:Young diagrams}
In this example $\Gamma$ is the alternating diagonal product associated to the group $\Z$. This is the classical B.H. Neumann group, 
see Example \ref{example:LEF approximation giving rise to BS} below for more details. Specifically, the subgroup $U$ is given by $\bigoplus_n \Alt(d_n)$ where $d_n$ is some strictly increasing sequence of odd integers greater than $5$. Fix a character $\varphi \in \Ch{U}$ so that $\varphi = \otimes_n \varphi_n$ with $\varphi_n \in \Ch{\Alt(d_n)}$.

\begin{itemize}
    \item Assume that the Young diagrams associated to the characters $\varphi_n$ are balanced, i.e. there is some constant $C > 0$ such that the Young diagrams have at most $C \sqrt{n}$ rows and columns. Then every non-trivial element $g$ is null for $\varphi$, see the main result of \cite{rattan2008upper}.
    
    \item Assume that $\varphi_n$ is the character of the standard $(d_n-1)$-dimensional representation on $\Alt(d_n)$. Recall that $\varphi_n(x) = \frac{d_n-1-|\mathrm{supp}(x)|}{d_n}$ for all $x \in \Alt(d_n)$. So an element $g \in G$ if null for $\varphi$ if and only if
    \[ \prod_{n \ge N} \frac{d_n-1-|\mathrm{supp}(\theta_n(g))|}{d_n} = 0\]
    for all $N$ sufficiently large. Hence if the sequence $d_n$ is \enquote{dense} as a subset of the integers, e.g. if $d_n = 2n+3$, then every non-trivial element $g \in G$ is null for $\varphi$, while if the sequence $d_n$ is \enquote{sparse}, namely $d_{n+1} \gg d_n$ for all $n$, then no  element $g \in G$  is null for $\varphi$. This example shows that the nullity condition can be a delicate one.
\end{itemize}
\end{example}

\subsection*{Inducing subgroups of diagonal products}

%Recall that $G$ is a finitely generated LEF group admitting local embeddings $\theta_n : A_n \to G_n$ for some sequence of pairwise non-isomorphic finite simple groups $G_n$. Further $\Gamma = \bigotimes_n G_n$ is the diagonal product associated to the $\theta_n$'s and to some fixed finite generating set $S$.
Let $N \lhd G$ be some normal subgroup. Denote $\Delta = t^{-1}(N) \lhd \Gamma$.

\begin{proposition}
\label{prop:vanishing for characters with infinite index in U}
Let $\varphi \in \Ch{U}$ be a character so that $\varphi = (\varphi_n)$ where $\varphi_n \in \Ch{G_n}$. Assume that 
\begin{itemize}
    \item $\left[U:\ker \varphi\right] = \infty$, and
    \item $\limsup |\varphi_n(\theta_n(g))| < 1$ for each element $g \in G \setminus N$, where the $\limsup$ is taken over all $n$ such that the character $\varphi_n$ is non-trivial.
\end{itemize}
Then all elements $g \in G \setminus N$ are null for $\varphi$. In particular, every character $\zeta \in \Ch{\Gamma}$ satisfying $\zeta _{\mid U} = \varphi$  vanishes outside of $\Delta$.
\end{proposition}
\begin{proof}
By definition, an element $g \in G $ is null for the character $\varphi$ if $\prod_{n=1}^\infty |\varphi_n(\gamma_n)| = 0$ for every element $\gamma = (\gamma_n) \in \Gamma$ with $t(\gamma) = g$. Every such  element $\gamma$ will satisfy $\gamma_n = \theta_n(g)$ for all $n$ sufficiently large. Therefore the first part of the statement follows  directly from the assumptions. The second part of the statement follows from the first, with the aid of Theorem \ref{thm:intro compatbile characters}, and more specifically the formula given in Theorem \ref{thm:traces over a character in a tail group}.
\end{proof}

%Recall that we have denoted $\Delta = t^{-1}(N)$.

\begin{proposition}
\label{prop:vanishing for characters with finite index in U}
Let $\zeta \in \Ch{\Gamma}$ be a character. Assume that  $\left[U:U \cap \ker \zeta\right] < \infty$ and $\left[\Delta:\Delta \cap \ker \zeta\right] = \infty$. If the subgroup $N$ of $G$  is inducing then $\zeta(\gamma) = 0$ for all elements $\gamma \in \Gamma \setminus \Delta$.
\end{proposition}
\begin{proof}
Denote $V = U \cap \ker \zeta$. By the algebraic Lemma \ref{lemma:algebraic structure of finite index subgroup of U} we have the isomorphism $\Gamma / V \cong G \times U/V$. The character $\zeta$ factors trough this direct product. As such $\zeta = \zeta_1 \otimes \zeta_2$ where $\zeta_1 \in \Ch{G}$ and $\zeta_2 \in \Ch{U/V}$, see Proposition \ref{prop:char_direct_prod}. The assumption that $\left[\Delta:\Delta \cap \ker \zeta\right] = \infty$ together with the fact that the quotient group $U/V$ is finite implies that the restriction of the character $\zeta_1$ to the subgroup $N$ of $G$ is non-trivial. As this subgroup is inducing, $\zeta_1$ vanishes outside of $N$. Hence the character $\zeta$ vanishes outside of $\Delta$, as required.
\end{proof}

The two Propositions \ref{prop:vanishing for characters with infinite index in U} and \ref{prop:vanishing for characters with finite index in U} combine to give the following conclusion.

\begin{theorem} \label{thm:weakly inducing}
Let $\Gamma = \bigotimes G_n$ be the  diagonal product associated to the  generating set $S$ and the essential LEF approximation $\theta_n : G \to G_n$ by simple groups. Let $t : \Gamma \to G$ be the tail map. Let $N$ be a normal subgroup of $G$ and denote $\Delta = t^{-1}(N)$. Assume that 
\begin{enumerate}
    \item the subgroup $N$ is inducing in $G$, and
    \item each element $g \in G \setminus N$ satisfies $$\limsup_n  \sup_{1 \neq \varphi \in \Ch{G_n}} |\varphi_n(\theta_n(g))| < 1.$$
\end{enumerate}
Then the subgroup $\Delta$ is weakly inducing.
\end{theorem}
\begin{proof}
Consider any character $\zeta \in \Ch{\Gamma}$ and assume that  $\left[\Delta : \Delta \cap \ker \zeta\right] = \infty$. Denote $\varphi = \zeta_{\mid U}  \in \Ch{U}$. There are two possibilities to consider, depending on the index of the  kernel of $\varphi$ inside $U$. If $\left[U:\ker \varphi\right] = \infty$ then $\zeta$ vanishes outside of $\Delta$ by Proposition \ref{prop:vanishing for characters with infinite index in U}, by relying on Assumption (2). Alternatively, if $\left[U:\ker \varphi\right] < \infty$ then Proposition  \ref{prop:vanishing for characters with finite index in U}  applies directly to the same effect.
\end{proof}

\subsection*{An application to characters of simple groups}

The proof of Theorem~\ref{thm:traces over a character in a tail group} has a shortcoming in that it depends  on the choice of the auxiliary trace $\widehat{\varphi}$. Interestingly, if the LEF group $G$ is simple (or admits a   simple normal subgroup $N$) then we can show a posteriori that this trace $\widehat{\varphi}$ is uniquely determined.

\begin{corollary}
\label{cor:about simple groups}
Let $\Gamma$ be a diagonal product with tail map $t : \Gamma \to G$.
Assume that  $  N \lhd G$ is a non-abelian simple locally inner normal subgroup. Denote $\Delta = t^{-1}(N)$. Fix an arbitrary character $\varphi = \otimes   \varphi_i \in \Ch{U}$. Then the sequence 
$$\widehat{\varphi}_n : g \mapsto  \prod_{i=1}^n \varphi_i(g_i)$$
is convergent in $\Tr{\Delta}$ and its limit is a character of $\Delta$.
\end{corollary}
\begin{proof}
Let $\widehat{\varphi}$ and $\widehat{\varphi}'$ be any two  accumulation points  of the sequence $\widehat{\varphi}_n$ in the (compact) space of traces $\Tr{\Delta}$. Both accumulation points satisfy $\widehat{\varphi}_{|U} = \widehat{\varphi}'_{|U} = \varphi$. Abusing notation, we let $\widehat{\varphi}$ and $\widehat{\varphi}'$ denote the trivial extensions from $\Delta$ to $\Gamma$. As the subgroup $\Delta \leq \Gamma$ is locally inner by Lemma \ref{lem:locally inner for diagonal products}, these trivial extensions are traces on $\Gamma$, see Proposition \ref{prop:locally-inner}.
According to Theorem \ref{thm:traces over a character in a tail group}  there is a pair of compatible traces $\psi,\psi' \in \Tr{G}$ such that
$$ \widehat{\varphi}' = \widehat{\varphi} \cdot (\psi \circ t) \quad \text{and} \quad \widehat{\varphi} = \widehat{\varphi}' \cdot (\psi' \circ t).$$
Substituting one equation into the other gives
$$ \widehat{\varphi}(g) = \widehat{\varphi}(g) \cdot (\psi \cdot \psi')(t(g)) \quad \forall g \in \Gamma.$$

We claim that $|\psi(\tau)| = 1$ for any element $\tau \in G$ which is not null for $\varphi$. Indeed,  for every such element $\tau$ there is by definition some $g \in \Gamma$ with $t(g) = \tau$ and $\widehat{\varphi}(g) \neq 0$. Recall that $\|\psi\|_\infty \le 1$ and $\|\psi'\|_\infty \le 1$. The claim follows.

There are now two separate  cases to consider:
\begin{itemize}
    \item All non-trivial elements $\tau \in N$ are null for the character $\varphi$. This  implies by compatibility that  $\psi = \delta_e = \psi'$ on $N$.
    \item There is some non-trivial element $\tau \in N$ which is not null for the character $\varphi$. Recall that $|\psi|^2 \in \Tr{G}$ \cite[Proposition 4.2]{lavi2023characters}. Therefore the normal subgroup  $\ker |\psi|^2 \lhd G$ is non-trivial by the above claim. Since the group $N$ of $G$ is simple we get $N \le \ker |\psi|^2 $ so that $|\psi| = 1$ on $N$. Therefore the trace $\psi$ is multiplicative on $N$, see Lemma \ref{lemma:traces of abs value 1}.  The non-abelian simple group $N$ is perfect.  We get $N \le \ker \psi$. In particular, no element of $N$ is null for $\varphi$.
\end{itemize}
In either case $\widehat \varphi = \widehat \varphi'$ on $\Delta$ so that the sequence $\widehat{\varphi}_n$ considered in the proof of Theorem \ref{thm:traces over a character in a tail group}   must be convergent in $\Tr \Delta$. In the first case $\widehat \varphi$ is the trivial extension of the character $\varphi$ from the subgroup $U$ to the subgroup $\Delta$. 
Since this is the only possible such extension,  we conclude that $\widehat{\varphi} \in \Ch{\Delta}$ - see the proof of Proposition \ref{prop:locally-inner and inducing}. In the second case, the accumulation point $\widehat{\varphi} \in \Ch{\Delta}$ corresponds to the trivial character $1 \in \Ch{G}$ via the bijective affine correspondence established in Theorem \ref{thm:traces over a character in a tail group}. As such $\widehat{\varphi}$ is a character of $\Delta$.
\end{proof}

We can now prove Theorem \ref{thm:cyclic structure} from the introduction.

\begin{proof}[Proof of Theorem \ref{thm:cyclic structure}]
Let $x$ be a fixed even permutation with $q = |\mathrm{\supp}(x)|$. Let $n_i \in \mathbb{N}$ be a strictly increasing sequence of integers with $n_1 \ge q$ and $\varphi_i \in \Ch{\Alt(n_i)}$ irreducible normalized characters. We first consider the following two special cases.
\begin{itemize}
    \item \emph{All $n_i$'s are odd}.
Let $\Gamma = \bigotimes_{i} \Alt(n_i)$ be the classical B.H. Neumann group. It is constructed in detail in Example \ref{example:LEF approximation giving rise to BS}.  This group admits a tail map $t : \Gamma \to \mathscr{A}(\mathbb{Z})$.  Consider the normal subgroups $N = \Altfin{\mathbb{Z}} \lhd \mathscr{A}(\mathbb{Z})$ and $\Delta = t^{-1}(N) \lhd \Gamma$. We have $U = \bigoplus \Alt(n_i) \lhd \Delta$.
Consider the character $\varphi = \otimes_i \varphi_i \in \Ch{U}$. With these objects at hand, Theorem \ref{thm:cyclic structure} follows immediately from Corollary \ref{cor:about simple groups}.
\item \emph{All $n_i$'s are even}. To deal with the even case, the above proof needs to be modified, since a full cycle is now an odd permutation. Consider the peculiar group $G = \mathbb{Z} \ltimes \Altfin{\mathbb{Z}}$ where the the action of the generator $\sigma_0$ of the infinite cyclic group $\mathbb{Z}$ on the normal subgroup $\Altfin{\mathbb{Z}}$ arises from the permutation $\widehat{\sigma}_0 \in \Alt(\mathbb{Z})$ given by
$$ 
\widehat{\sigma}_0(n) = \begin{cases}
n+1 & n \in \mathbb{Z} \setminus  \{-1,0\},\\
1 & n = - 1,\\
0 & n = 0.
\end{cases}
$$
In other words $\widehat{\sigma}_0$ is a full shift on the integers \enquote{skipping over} the point $0 \in \mathbb{Z}$. Further, consider $\tau = (-1,0,1) \in \Altfin{Z}$. The pair $\{\sigma_0,\tau\}$ generates the group $G$. There is a LEF approximation $\theta_i : G \to \Alt(n_i)$ in which $\theta_i(\sigma_0)$ is a $(n_i-1)$-cycle and $\theta_i(\tau)$ is a $3$-cycle for all $i$. From this point onwards, we may consider the corresponding diagonal product $\Gamma = \bigotimes_i \Alt(n_i)$ and proceed exactly as in previous  case. Note that Corollary \ref{cor:infinite orbit} applies so that $\Altfin{\mathbb{Z}}$ is an inducing subgroup of $G$.
\end{itemize}
In the general case, if all but finite many $n_i$'s are of the same parity, then the above argument still applies, for the conclusion of the theorem is not effected by ignoring finitely many $n_i$'s. If there infinitely many odd $n_i$'s as well as infinitely many even $n_i$'s, then the existence of the limit $\lim_{n \to \infty} \prod_{i=1}^n \varphi_i(x)$ in question follows by knowing it exists separately along odd and even $n_i$'s.
\end{proof}

\section{Approximating characters of alternating and elementary enrichments and diagonal products}
\label{sec:approximating}

In this section we take up the groups and constructions introduced in \S\ref{sec:diagonal products from alternating and elementary enrichments} and prove several technical approximation results, in anticipation for the study of Hilbert--Schmidt stability of these groups.

Let $G$ be a group with a finite generating set $S$ and a fixed LEF approximation $\theta_n : G \to G_n$ by groups with pairwise distinct orders and satisfying the parity condition. 

Let $\Gamma_0 = G \ltimes N$ be either  the alternating enrichment ${\mathscr A(G)}$ or the elementary enrichment ${\mathscr E(G)}$ of the group $G$ with its fixed set of generators $S \amalg T$ (see Lemmas \ref{lemma:encoding into alternating} and \ref{lemma:encoding into elementary}, respectively). Here the normal subgroup $N$ is either $N = \Altfin{G}$ or $N = \SLfin{V_G}$, depending on the case. The subgroup $N$  is normally generated by the set  $T$.

Let $\Gamma = \bigotimes H_n$ denote the associated alternating (or elementary) diagonal product arising from the LEF approximation $\Theta_n$ of the enrichment $\Gamma_0$ into the sequence of finite groups $H_n = \Alt(G_n)$ or  $H_n = \mathrm{SL}(V_{G_n})$ (see Definitions \ref{def:enriched alternating diagonal product} and \ref{def:elementary enriched diagonal product}, respectively).
 It will be useful to denote
$$\Delta = t^{-1}(N) \le \Gamma$$
where  $t : \Gamma \to \Gamma_0$ is the tail map.

We may  assume for the remainder of \S\ref{sec:approximating}  that the group $G$ is infinite, for otherwise its enrichment $\Gamma_0$ and the associated diagonal product $\Gamma$ are both finite. In particular $|G_n| \to \infty$ as $n\to\infty$. We may  rearrange the order of terms in the LEF approximation $\theta_n : G \to G_n$ without changing the isomorphism type of the  diagonal product $\Gamma$ and assume without loss of generality that $|G_n| \le |G_{n+1}|$ holds true for all $n \in \mathbb{N}$. 
%TODO: \ad{I am not sure exactly what us meant by rearanging the terms in the LEF approximation, and why it implies the isomomrphism type doesn't change. I suggest we just assume monotonicity is the case for all LEF approximations in this article, as it is not interesting to consider ones which are not monotone}

Choose an arbitrary sequence of injective set-theoretical maps  $f_n : G_n \to G_{n+1}$ subject to the condition that $\theta_{n+1}(g) = f_n \circ \theta_n(g)$ for each fixed element $g \in G$ and for all $n$ sufficiently large (depending on $g$). Let $h_n : H_n \to H_{n+1}$ be the sequence of group monomorphisms obtained by using the sequence of maps $f_n$ to identify the underlying set of the group $G_n$ as a subset of the underlying set of the group $G_{n+1}$ in the alternating case, and the vector space $V_{G_n}$ as a subspace of the vector space $V_{G_{n+1}}$ in the elementary case.  
%TODO: \ad{should explain that one can indeed find $f_n$ since $\theta_n$ is asymptotically injective, so there are no conflicts for different $g \in G$}

Observe that the LEF approximation  $\Theta_n : \Gamma_0 \to H_n$ given by \enquote{encoding into alternating/elementary groups}  satisfies $\Theta_{n+1}(x) = h_n \circ \Theta_n(x)$ for each element $x \in N$ and all $n$ sufficiently large.

\begin{proposition}\label{prop:Approximating traces on N}
For each trace  $\varphi \in \Tr{\Delta}$ there is a sequence of traces $\varphi_k \in \Tr{\Delta}$ with $\left[\Delta:\ker \varphi_k \right] < \infty$ for all $k \in \N$ such that $\varphi_k \to \varphi$ pointwise on $\Delta$.
\end{proposition}

\begin{proof}

It will suffice to construct a sequence of group endomorphisms $\pi_k:\Delta\to \Delta$ with the following two properties:
    \begin{enumerate}
        \item $\left[\Delta:\ker \pi_k\right] < \infty$ for each $k \in \N$, and
        \item For each element $x\in \Delta$ we have  $\pi_k(x)=x$  for all but finitely many $k$'s.
    \end{enumerate}
    Indeed, given such a sequence of endomorphisms $\pi_k$, for every  trace $\varphi \in \Tr{\Delta}$ on the tail group $\Delta$ the sequence of traces $\varphi_k=\varphi\circ \pi_k$ is as required.

    We now construct a sequence of endomorphisms $\pi_k$  satisfying the above two properties.  Consider the quotient map  $p_k :   \Delta \to    \prod_{i=1}^k H_i$ as well as the embedding 
\begin{equation}
\label{eq:iota}
    \iota_k : \prod_{i=1}^k  H_i \to    \Delta, \quad \iota_k: (g_1,...,g_k)\mapsto (g_1,..,g_k,h_k(g_k),(h_{k+1} \circ h_k)(g_k),...)
\end{equation}
for each $k$.
The desired  endomorphisms are given by $\pi_k = \iota_k \circ p_k : \Delta \to \Delta$. These satisfy $\ker \pi_k = \ker p_k$ so that $\left[\Delta:\ker \pi_k\right] < \infty$ for all $k$. We get property (1). Finally, property (2) follows from the fact that $h_n \circ \Theta_n(x) = \Theta_{n+1}(x)$ for each fixed element $x \in N$ and for all $n$ sufficiently large.
\end{proof}

Building upon the  above argument, we obtain  the following approximation result for characters of enrichments.

\begin{proposition}
\label{prop:Approximating traces on G via N}
 Recall that $\Gamma_0 = G \ltimes N$ is the alternating (or elementary) enrichment of the group $G$. Let $\psi \in \Ch{\Gamma_0}$ be any character with $\psi_{|N}  \neq 1$. Then there is a sequence of traces $\zeta_n \in \Tr{H_n}$ such that the functions  $\zeta_n \circ \Theta_n : \Gamma_0 \to \mathbb{C}$ converge pointwise to the character $\psi$ on $\Gamma_0$.
\end{proposition}
\begin{proof}
We make use of the  
 map $\iota_n : \prod_{i=1}^n H_i \to \Delta $ constructed in Equation (\ref{eq:iota}) of the previous   Proposition  \ref{prop:Approximating traces on N}. Let $q_n : H_n \to \prod_{i=1}^n H_i$ be the obvious embedding in the $n$-th  coordinate.  Note that the sequence of set-theoretical maps 
 $$t \circ \iota_n \circ q_n \circ \Theta_n : N \to N$$ converges pointwise to the identity map on the subgroup $N$. 

We use the tail map $t : \Gamma \to  \Gamma_0$ to pullback the character $\psi$ of the group $\Gamma_0$ to the character $\varphi = \psi \circ t$ of the group $\Gamma$. Denote 
$$\zeta_n = \varphi \circ \iota_n \circ q_n = \psi \circ t \circ \iota_n \circ q_n$$ so that $\zeta_n \in \Tr{H_n}$. The fact that $\zeta_n \circ \Theta_n$ converges pointwise to the character $\psi$ on the subgroup $N$ follows from the above remarks.

It remains to deal with elements of $\Gamma_0$ not belonging to the normal subgroup $N$. Recall that the subgroup $N$ is inducing (see Corollaries \ref{cor:chars of alternating enrichments} and  \ref{cor:chars of elementary enrichments} for the alternating and elementary cases, respectively). Therefore $\psi(g) = 0$ for all $g \in \Gamma_0 \setminus N$. The fact that $\zeta_n \circ \Theta_n(g) \to 0$ for every such element $g \in \Gamma_0 \setminus N$ follows immediately from Propositions \ref{prop:alt infinity is ranked} and \ref{prop:infinite permutations are ranked} in the alternating and elementary cases, respectively.
\end{proof}

\mypart{Hilbert--Schmidt stability}

\section{Stability of diagonal products}
\label{sec:HS_stab}

In this section we study Hilbert--Schmidt stability as well as local Hilbert--Schmidt stability for families of groups that were considered in previous sections, including diagonal products and enrichments. In particular, we prove Theorems \ref{intro:thm:arbitrary large SRad}, \ref{intro:thm:diag_prod_stab_main}, \ref{thm:intro:local stability passes to enrichments} from the introduction.
 
We recall the elegant and useful   Hadwin--Shulman criterion for Hilbert--Schmidt stability of amenable groups, which forms the basis of all our stability results.

\begin{theorem}[Hadwin-Shulman \cite{HS_grp}]\label{thm:hadwin-shulman}
An amenable countable group $G$ is Hilbert--Schmidt stable if and only if
any trace on $G$ is a pointwise limit of finite-dimensional traces. 
\end{theorem}

To verify the Hadwin--Shulman criterion it suffices to consider characters of the group $G$ (rather than all traces on $G$), see \cite[Theorem 4]{HS_grp}.

\subsection*{Stability of diagonal products}

We now prove Theorem \ref{intro:thm:diag_prod_stab_main} of the introduction, which says for amenable groups, that diagonal products preserve Hilbert--Schmidt stability (given  some technical assumptions on the LEF approximation in questioon). Actually, the proof follows quite easily and directly from our character classification result for diagonal products, i.e. Theorem \ref{thm:traces over a character in a tail group}.

\begin{proof}[Proof of Theorem \ref{intro:thm:diag_prod_stab_main}]
Let $G$ be an amenable Hilbert--Schmidt stable group with a fixed finite generating set $S$ and an essential LEF approximation  $\theta_n : G \to G_n$ by simple groups (see Definitions \ref{def:partial homo and LEF approx} and \ref{def:essential LEF}). Let $\Gamma = \bigotimes_n G_n$ be the diagonal product associated to this data. We want to show that the amenable group $\Gamma$ is Hilbert--Schmidt stable by relying on the Hadwin--Shulman criterion.

Consider an arbitrary character $\zeta \in \Ch{\Gamma}$.
Let $U \lhd \Gamma$ be the kernel of the tail map $t : \Gamma \to G$, so that $U \cong \bigoplus G_n$ by Lemma \ref{lem:sum_contained_in_diag_prod}. We know that $\varphi = \zeta_{\mid U} = \otimes_n \varphi_n$, where $\varphi_n \in \Ch{G_n}$ for all $n$ (see the discussion in \S\ref{sec:LEF groups and diagonal products} and \S\ref{sec:diagonal products}). According to our character classification result (Theorem \ref{thm:traces over a character in a tail group}), there is trace $\psi \in \Tr{G}$ 
which is compatible  with $\varphi$ in the sense of Definition \ref{def:compatible trace} as well as an auxiliary trace  $\widehat{\varphi} \in \Tr{\Gamma}$ such that
$$ \zeta(g) = \widehat{\varphi}(g) \cdot \psi(t(g)) \quad \forall g \in \Gamma.$$
The auxiliary trace $\widehat{\varphi}$ is a pointwise limit of a some sequence $\widehat{\varphi}_{k}$, where each $\widehat{\varphi}_k \in \Ch{\Gamma}$ is a trace on $\Gamma$ factoring through the finite group $U_k = \prod_{i=1}^{n_k} G_i$ for some $n_k$, see the proof of Theorem \ref{thm:traces over a character in a tail group}. In particular, each such $\widehat{\varphi}_k$ is a finite-dimensional trace on $\Gamma$.

The Hadwin--Schulman criterion (Theorem \ref{thm:hadwin-shulman}) applies to the amenable and Hilbert--Schmidt stable group $G$. In particular,  there exists a sequence of finite-dimensional traces $\psi_k \in \Tr{G}$ converging pointwise to the trace $\psi$. 
The function $\zeta_k : \Gamma \to \mathbb{C}$ defined for each $k$ by
$$\zeta_k = \widehat{\varphi}_{k}(g) \cdot \psi_k(t(g)) \quad \forall g \in \Gamma$$ 
is a  finite-dimensional traces on the group $\Gamma$ (see Proposition \ref{prop:prod fin dim traces}). The sequence of traces $\zeta_k$ clearly converges pointwise to the given character $\zeta$. As the character $\zeta$ was arbitrary, we conclude  that the amenable group $\Gamma$ is Hilbert--Schmidt stable from  Theorem \ref{thm:hadwin-shulman}.
\end{proof}
 
\subsection*{Stability of alternating and elementary diagonal products}

We now prove that alternating and elementary diagonal products (see Definitions \ref{def:enriched alternating diagonal product} and \ref{def:elementary enriched diagonal product}, respectively) of a group $G$ are Hilbert--Schmidt stable, under the assumption that every finite extension $G$ is stable (see Theorem \ref{thm:alt diag prod is stable} below).

We begin with a general criterion where Hilbert--Schmidt stability is inherited from a locally inner subgroup of an amenable group.
Recall that the notion of a weakly inducing subgroup was given in Definition \ref{def:inducing and weakly inducing}.

\begin{proposition}
 \label{prop:going from N HS to G HS}
Let $G$ be an amenable group. Let $N$ be a locally inner weakly inducing normal subgroup of $G$. Assume that
 \begin{itemize}
     \item every trace of $N$ is the pointwise limit of traces with finite-index kernels, and
     \item every  finite extension\footnote{Here by finite extension we mean any group surjecting onto $G/N$ with finite kernel.} of the group $G/N$ is Hilbert--Schmidt stable.
 \end{itemize}
 Then the group $G$ is Hilbert--Schmidt stable.
\end{proposition}
\begin{proof}
Consider an arbitrary character $G \in \Ch{G}$. There are two cases to consider.

In the first case $\left[N:N \cap \ker \varphi\right] < \infty$. This means that $\varphi$ factors through a character of the quotient $G_\varphi = G/(N \cap \ker \varphi)$. The group $G_\varphi$ is a finite extension of the quotient $G/N$ and is therefore Hilbert--Schmidt stable by assumption. It follows from Theorem \ref{thm:hadwin-shulman} that $\varphi$ can be approximated by finite-dimensional traces of the group $G$ (factoring through $G_
    \varphi$).

In the second case $\left[N:N \cap \ker \varphi\right] = \infty$. In this situation, the assumption that the normal subgroup $N$ is weakly inducing implies that $\varphi(g) = 0$ for all elements $g \in G \setminus N$. The other assumption of the proposition allows us to approximate the restriction $\varphi _{\mid N}$ by a sequence of traces $\psi_n \in \Tr{N}$ with $\left[N: \ker \psi_n\right] < \infty$. The fact that the normal subgroup $N$ is locally inner implies that the trivial extensions $\widetilde{\psi}_n$ from $N$ to $G$ are traces, see Proposition \ref{prop:locally-inner}. At this point, we conclude by arguing exactly as in the first case, in order to approximate each such trivial extension $\widetilde{\psi}_n$ by finite-dimensional traces on the group $G$.
\end{proof}

%Let $\Gamma = \bigotimes G_n$ be the associated alternating (or elementary) enriched diagonal product. Note that $\Gamma$ is amenable.

\begin{theorem}
\label{thm:alt diag prod is stable}
Let $G$ be an amenable LEF group equipped with a finite generating set $S$ and an essential LEF approximation  $\theta_n : G \to G_n$ by simple groups.
Assume that every  finite extension of the group $G$ is Hilbert--Schmidt stable. Then the associated alternating (or elementary) diagonal product $\Gamma$  is Hilbert--Schmidt stable.
\end{theorem}

\begin{proof} 
%To Alon: the theorem you mentioned should not be invoked directly
Let $\Gamma_0 = G \ltimes N$ denote the alternating enrichment ${\mathscr A(G)}$ or the elementary enrichment ${\mathscr E(G)}$ of the group $G$. Here the normal subgroup $N$ is either $\Altfin{G}$ or $\SLfin{V_G}$. The associated alternating or elementary diagonal product  is given by $\Gamma = \bigotimes H_n$ where either $H_n = \Alt(G_n)$ or $H_n = \mathrm{SL}(V_{G_n})$. See \S\ref{sec:diagonal products from alternating and elementary enrichments} for details on these constructions. In any case, the diagonal product  $\Gamma$ is amenable.

We will establish that the diagonal product $\Gamma$ is Hilbert--Schmidt stable by verifying all the necessary assumptions in Proposition \ref{prop:going from N HS to G HS}.
Recall that there is a tail map $t : \Gamma \to \Gamma_0$. Denote $\Delta = t^{-1}(N) $ so that $\Delta$ is a normal subgroup of $\Gamma$. 
\begin{itemize}
\item 
The first isomorphism theorem gives $\Gamma/\Delta \cong \Gamma_0/N \cong G$. This means that every finite extension of the quotient group $\Gamma/\Delta$ is Hilbert--Schmidt stable.
\item The normal subgroup $\Delta$ is locally inner by Lemma \ref{lem:locally inner for diagonal products}, since $N$ is a locally inner subgroup of $\Gamma_0$ in both the alternating and elementary case.
\item We claim that $N$ is a weakly inducing subgroup of $\Gamma$. This claim will be deduced from 
Theorem \ref{thm:weakly inducing} by verifying all of its necessary assumptions. 
Condition (1) of Theorem \ref{thm:weakly inducing} says that the normal subgroup $N$ is inducing is $\Gamma_0$. This holds true  by Corollary \ref{cor:chars of alternating enrichments} in the alternating case, and by Corollary \ref{cor:chars of elementary enrichments} in the elementary case.

Let us consider the bound in Condition (2) of Theorem \ref{thm:weakly inducing}. In the elementary case, it follows  from the uniform character bound of Theorem \ref{thm:character bounds for SLn}.
In the alternating case, observe that each element $g \in \Gamma_0 \setminus N$ can be written as $g = x  h$ for some non-trivial element $x \in G$ and some element $h \in N = \Altfin{G}$.
Consequently $\theta_n(g) = \theta_n(x) \theta_n(h)$ for all sufficiently large $n$, and $\theta_n(x)$ is a fully supported permutation while $| \supp(\theta_n(h)) | \leq M$ for some fixed $M \in \N$. Therefore  $| \supp(\theta_n(g)) | \geq |G_n| - M$ for all $n$ sufficiently large.
The desired Condition (2) of Theorem \ref{thm:weakly inducing} now follows by the character bound in Lemma \ref{lem:A_n_char_bound}.
\item
The last condition to check is that 
every trace on the normal subgroup $\Delta$ is a pointwise limit of traces on $\Delta$ with finite-index kernels. In turn, this follows from Proposition \ref{prop:Approximating traces on N}. 
\end{itemize}
Having verified all the requirements of Proposition \ref{prop:going from N HS to G HS}, we   conclude that the diagonal product $\Gamma$ is indeed Hilbert--Schmidt stable, as required.
\end{proof}

\subsection*{Local 
stability of alternating and elementary enrichments}

Local  Hilbert--Schmidt stability was introduced in Definition \ref{def:local stability}. The following  is a characterisation of local Hilbert--Schmidt stability  of amenable groups due to \cite{FFGS}. It is reminiscent of the above-mentioned Hadwin--Shulman criterion.

\begin{theorem}[Theorem B of \cite{FFGS}]\label{thm:FFGS}
Let $G$ be a countable amenable  group. Then the following two conditions are equivalent:
\begin{enumerate}
    \item The group $G$ is locally Hilbert--Schmidt stable.
    \item For every trace $\varphi \in \Tr G$  there exists a sequence of partial homomorphisms $\kappa_n:G \to U(d_n)$ for some sequence $d_n \in \mathbb N$, such that the functions $\frac{1}{d_n} \tr \circ \kappa_n$ converge pointwise to the trace $\varphi$.
\end{enumerate}
\end{theorem}

We note that  it suffices to verify condition (2) only for characters of the group $G$, rather than for all traces on $G$. For this, see \cite[Lemma 4.5]{FFGS}.
In addition, we note the following general fact.

\begin{proposition}
\label{prop:local stability goes down quotients}
Let $G = Q \ltimes N$ be a countable amenable semi-direct product. If $G$ is locally Hilbert--Schmidt stable then so is $Q$.
\end{proposition}
\begin{proof}
Let $q: G \to Q$ denote  the quotient map and $r: Q \to G$ be the inclusion map. We will make use of the criterion in Theorem \ref{thm:FFGS}. Let $\varphi \in \Tr{Q}$ be any trace. Then $\varphi \circ q \in \Tr{G}$. Assuming that  the group $G$ is locally Hilbert--Schmidt stable, there is a sequence of partial homomorphisms $\kappa_n : G \to U(d_n)$ such that the functions $\frac{1}{d_n} \tr \circ \kappa_n$ converge pointwise to $\varphi \circ q$. Then the sequence of partial homomorphisms $\kappa_n \circ r : Q \to U(d_n)$ has the functions  $\frac{1}{d_n} \tr \circ \kappa_n \circ r$  converging pointwise to the given trace $\varphi$, so that the group $Q$ is locally Hilbert--Schmidt stable.
\end{proof}

\begin{remark}
Proposition \ref{prop:local stability goes down quotients} is true without the amenability assumption, as can be seen  by using the original definition of local Hilbert--Schmidt stability of \cite{FFGS}.
\end{remark}

We now deal with local Hilbert--Schmidt stability of enrichments.
 
\begin{theorem}
\label{thm:enrichments are locally hs stable}
%Let $G$ be a countable amenable LEF   group. If the group $G$ is locally Hilbert--Schmidt stable then its alternating enrichment $\mathscr A(G)$ and its elementary enrichment $\mathscr E(G)$ are  also both locally Hilbert--Schmidt stable.
Let $G$ be a finitely generated amenable group. Then $G$ is locally Hilbert--Schmidt stable if and only if its  alternating enrichment $\mathscr A(G)$ is locally Hilbert--Schmidt stable. The exact same statement applies to the elementary enrichment $\mathscr E(G)$.
\end{theorem}
\begin{proof} 
Let $\Gamma_0$ denote either the alternating enrichment $\mathscr A(G)$ or the elementary enrichment $\mathscr E(G)$ of the group $G$. This means that $\Gamma_0 = G \ltimes N$ where $N = \Altfin{G}$  in the alternating case and $N = \SLfin{V_G}$ in the elementary case.

We start by assuming that the group $G$ is locally Hilbert--Schmidt stable and show that so is the enrichment $\Gamma_0$. An amenable local Hilbert--Schmidt stable group is LEF \cite[Lemma 3.12]{FFGS}. It follows that the group $G$ as well as its enrichment $\Gamma_0$ are LEF, see \cite{Bradford_LEF_ext}. 

Let $\varphi \in \Ch{\Gamma_0}$ be an arbitrary character. There are two separate cases to consider.
In the first case,  assume that $\varphi$ factors through the quotient $G$ of the alternating or elementary enrichment $\Gamma_0$. Then $\varphi$  can be approximated as required in item (2) of Theorem \ref{thm:FFGS}, from the assumption that $G$ is locally Hilbert--Schmidt and by composing with the quotient map from $\Gamma_0$    to $G$.

 In the second case, assume that the character $\varphi$ has a non-trivial restriction to the normal subgroup $N$. Proposition \ref{prop:Approximating traces on G via N} provides us with a sequence of finite groups $H_n$, a sequence of traces $\zeta_n \in \Tr{H_n}$ and a LEF approximation $\Theta_n : \Gamma_0 \to H_n$ such that the functions $\zeta_n \circ \Theta_n$ converge to the character $\varphi$ pointwise on the group $\Gamma_0$. For each $n$ there is some dimension $d_n$ and some unitary representation $\pi_n : H_n \to U(d_n)$ such that $\zeta_n = \frac{1}{d_n} \tr \circ \pi_n$. Denote $\kappa_n = \pi_n \circ \Theta_n$. 
 
 The fact that the group $\Gamma_0$ is locally Hilbert--Schmidt stable now follows from Theorem \ref{thm:FFGS}.
The converse direction is true in an even greater generality by Proposition \ref{prop:local stability goes down quotients}.
\end{proof}

We can now provide a different proof of \cite[Theorem E]{FFGS}, as mentioned in the introduction:

\begin{corollary}
\label{cor:uncountably many locally not HS stable}
There exist uncountably many pairwise non-isomorphic $4$-generated locally Hilbert--Schmidt stable groups which are  not Hilbert--Schmidt stable.
\end{corollary}

\begin{proof}
Let $\mathcal{P}$ and $\mathcal{P}'$ be two distinct strictly increasing sequences of odd integers greater than $4$. Let $\Gamma$ and $\Gamma'$ be the two associated classical B.H. Neumann groups, as in Example \ref{example:LEF approximation giving rise to BS}. The two groups $\Gamma$ and $\Gamma'$ are non-isomorphic  \cite{Neu} (see also \cite[\S2]{LL2}).
Even more, it is known that the profinite completions of those two groups are non-isomorphic \cite[Proposition 4.1]{LPS}. We claim that the alternating enrichments $\mathscr A(\Gamma)$ and $ \mathscr A(\Gamma')$ are also non-isomorphic. 
Indeed, this claim follows from the fact that every finite quotient of the group $\mathscr A(\Gamma)$ factors through $\Gamma$, since $\mathscr A(\Gamma) = \Gamma \ltimes \Altfin{\Gamma}$ and $\Altfin{\Gamma}$ is an infinite simple group.
Finally, every such alternating enrichment $\mathscr A(\Gamma)$ is 
locally Hilbert--Schmidt stable by Theorem \ref{thm:enrichments are locally hs stable}, but 
not Hilbert--Schmidt stable, for $\mathscr A(\Gamma)$ is not residually finite \cite[Proposition 2.5]{DogonAlon2021SaAo}. The B.H. Neumann groups $\Gamma$  are $2$-generated, and hence their alternating enrichments $\mathscr A (\Gamma)$  are $4$-generated, see Lemma \ref{lemma:encoding into alternating}.
\end{proof}

\section{Stability radius growth}\label{sec:stab rad}

In this section we study some basic properties of the stability radius growth function (Definition \ref{def:SRad}).
Of particular interest is the relation between stability radius growth and the LEF and MAP growth functions of a group (Proposition \ref{prop:SRad_LEF_bound}).
We begin by showing that stability radius growth is an uninteresting notion for finitely presented groups.

\begin{lemma} \label{lem:SR_fin_presented}
If the group $\Gamma$ is finitely presented and Hilbert--Schmidt stable then the function $\SR_\Gamma^S$ is bounded.
\end{lemma}

\begin{proof}
Let $\F_S$ be the free group with generators labelled by $S$. There is a natural quotient map  $p: \F_S \to \Gamma$. If the group $\Gamma$ is finitely presented, then the kernel $\ker p$ is normally generated by a finite set of relations $R \subset \F_S$.

 The Hilbert--Schmidt stability of the group $\Gamma$ is  equivalent to the following \cite{HS_grp}: For every $M> 0$ there is a $\delta > 0$ such that for every $d \in \N$ and every homomorphism $\varphi: \F_S \to U(d)$ with $d_\mathrm{HS}(\varphi(w),\mathrm{Id}) \leq \delta$ for all $w \in R$, there exists a homomorphism $\pi: \mathbb F_S \to U(d)$ factoring through $\Gamma$ and satisfying $d_\mathrm{HS}( \varphi(s), \pi(s) ) \leq \frac{1}{M}$ for all generators $s \in S$.
Consequently, if we let $l \in \mathbb{N}$ be the maximal word length of any relation $w\in R$, then $\SR_\Gamma^S(M) \leq l$ for all $M>0$.
\end{proof}

\begin{remark}
It is natural to ask whether the converse direction of Lemma \ref{lem:SR_fin_presented} is true. Namely, if $\SR_\Gamma^S$ is bounded, does it follow that $\Gamma$ is finitely presented? 

For example, this implication would hold  if all groups were hyperlinear, as we now explain. Let $\Gamma$ be a group with a finite generating set $S$ and  $p: \F_S \to \Gamma$ be the corresponding quotient map.
Assume that $\SR_\Gamma^S$ is bounded, namely  $\SR_\Gamma^S \le l$ for some $l \in \mathbb{N}$. Consider the quotient group  $\Gamma_l$ of the free group $\F_S$ presented by the subset of all relations of length at most $3l$   belonging to $\ker p$. Let  $p_l: \F_S \to \Gamma_l$  be the corresponding quotient map. 
We claim that the two groups $\Gamma_l$ and $\Gamma$ coincide, which in particular implies that $\Gamma$ is finitely presented.

Assume by contradiction that there is a word $w_0 \in \F_S$ with $w_0 \in \ker p \setminus \ker p_l$.
Let $\delta > 0$  be the parameter   provided by Definition \ref{def:SRad} with respect to the value $\SR_\Gamma^S(|w_0|)$. 
Since the group $\Gamma_l$ is hyperlinear (by our standing assumption),   there exists a map $\varphi: \F_S \to U(d)$ which satisfies $d_\textrm{HS}(\varphi(w),\mathrm{Id}) < \delta$ for all relations $w \in \ker p$ of length at most $3l$, and such that $d_\mathrm{HS}(\varphi(w_0),\mathrm{Id}) > 1$.
The map $\varphi$ descends to an $(l,\delta)$-almost homomorphism $\overline{\varphi} : \Gamma \to U(d)$ such that $\varphi(s)=(\overline{\varphi} \circ p)(s)$ for all $s \in S$.
Since $\SR_\Gamma^S(|w_0|)  \leq l$, there exists a homomorphism $\pi: \F_S \to U(d)$ which factors through $p$ and satisfies $d_\mathrm{HS}( \varphi(w_0), \pi(w_0)) < 1$.  This is a contradiction to the fact that $\pi(w_0) = \mathrm{Id}$.
\end{remark}
%\marginpar{To Alon: you wrote does residual finiteness suffice. However every group with SR finite, is HS-stable, hence residually finite.}

Recall the notion of equivalence between monotone functions which was a part of  our standing notations on p. \pageref{notations}.

\begin{lemma}\label{lem:SR_independent_of_gen}
If $S_1$ and $S_2$ is a pair of finite generating sets for the group $\Gamma$ then $\SR_\Gamma^{S_1} \approx \SR_\Gamma^{S_2}$.
\end{lemma}
\begin{proof}
Let $l  \in \mathbb{N}$ be radius  such that $S_1 \subset B_{S_2}(l)$ and $S_2 \subset B_{S_1}(l)$. Note that for all $n, d\in \mathbb{N} $ and  $\delta > 0$, if $\varphi: B_{S_2}(n l) \to U(d)$ is an $(nl,\delta)$-almost representation (with respect to $S_2$) then the restriction of $\varphi$ to $B_{S_1} (n) \subset B_{S_2}(n l)$ is a $(n,\delta)$-almost representation (with respect to $S_1$).
Further, if $\pi: \Gamma \to U(d)$ is a homomorphism with $\| \varphi(s_1) - \pi(s_1) \|_2 \leq 1/M$ for all $s_1 \in S_1$ and for some $M > 0$, then a calculation shows $\| \varphi(s_2) - \pi(s_2) \|_2 \leq (1/M+\delta)l$ for all $s_2 \in S_2$.
Since the parameter $\delta$ can be taken to be arbitrarily small, we have $\SR_\Gamma^{S_2}(M) \leq l \cdot \SR_\Gamma^{S_1}((l + 1)M)$. A similar inequality holds in the opposite direction by symmetry, and we are done.
\end{proof}

In light of Lemma \ref{lem:SR_independent_of_gen}, we may denote by $\SR_\Gamma$ the growth-type equivalence class of the stability radius growth function  $\SR_\Gamma^S$ for some arbitrary generating set $S$.  The growth-type equivalence class $\SR_\Gamma$ is a group invariant.

\begin{remark}
This entire discussion  can be readily adapted to the setting of  stability in permutations, or even to a more general setting of stability with respect to a given family of metric groups.
\end{remark}

\subsection*{Quantitative approximation properties}

We recall various notions of quantification for group approximation properties, following  \cite{Bradford_LEF,BS}, and relate them to stability radius growth.

\begin{definition}\label{def:approximation_growth} 
Let $\Gamma$ be a group equipped with a fixed finite generating set $S$.
\begin{itemize}
\item
The \emph{full residual finiteness growth} of the group $\Gamma$ is the function 
$$\RF_\Gamma^S: \N \to \N\cup \{\infty\}$$  where for each $n \in \mathbb{N}$ the value   $\RF_\Gamma^S(n)$ is the smallest order $|F|$ of a finite quotient $q : \Gamma \to F$ injective on the ball $B_S(n)$.
\item
The \emph{LEF growth} of the group $\Gamma$ is the function $$\LEF_\Gamma^S: \N \to \N\cup \{\infty\}$$
where for each $n \in \mathbb{N}$ the value
 $\LEF_\Gamma^S(n)$  is the smallest order $|F|$ of a finite group $F$ admitting a \emph{local embedding} $\varphi: B_S(n) \to F$ (i.e. an injective map satisfying $\varphi(gh) = \varphi(g) \varphi(h)$ for every pair of elements $g,h\in B_S(n)$ with $gh \in B_S(n)$).

\item
The \emph{MAP growth} of the group $\Gamma$ is the function 
$$\MAP_\Gamma^S: \N \to \N\cup \{\infty\}$$ 
where for each $n \in \mathbb{N}$ the value
$\MAP_\Gamma^S(n)$  is the smallest $d \in \mathbb{N}$ such that there is a unitary representation $\pi: \Gamma \to U(d)$ such that $\pi$ injective on the ball $B_S(n)$. 
\end{itemize}
If no finite value as required exists, then the value of any of the above functions is defined to be  $\infty$.
\end{definition}

Clearly $\RF^S_\Gamma(n), \LEF^S_\Gamma(n)$ or $\MAP^S_\Gamma(n)$ is finite for all $n \in \mathbb{N}$ if and only if the group $\Gamma$ is residually finite, LEF or MAP, respectively.

Upon passing from one finite generating set to the other, 
the three functions $\RF_\Gamma^S, \LEF_\Gamma^S$ and $\MAP_\Gamma^S$ stay in the same growth-type equivalence class. Hence it makes sense to talk of the three growth-type classes $\RF_\Gamma, \LEF_\Gamma$ and $\MAP_\Gamma$. These are  invariants of the group $\Gamma$. See \cite{Bradford_LEF} for details.

We now relate the stability radius, the MAP and the LEF growth. This can be thought of as a quantitative version of the following folklore observation: a Hilbert--Schmidt stable hyperlinear group  is MAP (see \cite[Proposition 2.5]{DogonAlon2021SaAo}).

\begin{proposition}
\label{prop:SRad_LEF_bound}
        Let $\Gamma$ be a Hilbert--Schmidt stable LEF group generated by the finite set $S$. Then for all $n \in \N$ the following bound holds 
         $$ \MAP_\Gamma^S(n) \leq \LEF_\Gamma^S(\max\{\SR_\Gamma^S(n), n\}).$$
\end{proposition}
\begin{proof}
Fix $n \in \mathbb{N}$. Denote  $m = \max\{\SR_\Gamma^S(n), n\}$. There is a local embedding $\varphi: B_S(m) \to F$ into some finite group  $F$ of order $\LEF_\Gamma^S(m)$.
By composing $\varphi$ with the left-regular representation of the group $F$, we get a local embedding $$\psi: B_S(m) \to U(|F|)$$ with the property that $d_\mathrm{HS}( \psi(g), \mathrm{Id}) = \sqrt{2}$ for all $g \in B_S(m) \setminus \{e\}$. Here $U(|F|)$ is the unitary group of dimension equal to $|F| \in \mathbb{N}$.
As $m \ge \SR_\Gamma^S(n)$,  there exists a unitary representation $\pi: \Gamma \to U(|F|)$ with $$d_\mathrm{HS}( \pi(s), \psi(s) ) \leq  1/n$$ for all generators $s \in S$.
    Consequently, the triangle inequality and the fact that $m \ge n$ implies  $d_\mathrm{HS}( \pi(g), \psi(g) ) \leq 1$ for all elements $g \in B_S(n)$.
    As such $$d_\mathrm{HS}( \pi(g), \mathrm{Id}) \geq \sqrt{2} - 1 > 0$$ for all elements $g\in B_S(n) \setminus \{e\}$. In other words, the unitary representation  $\pi$ is injective on the ball $B_S(n)$. We conclude that $$\MAP_\Gamma^S(n) \le |F| = \LEF_\Gamma^S(m)$$
    as required.
\end{proof}

\subsection*{Groups with arbitrarily large stability radius growth}

We now explain how Theorem \ref{intro:thm:diag_prod_stab_main}, coupled with the work of Bradford \cite{Bradford_LEF} and the generalised B. H. Neumann groups from Section \ref{sec:generalised Neumann groups} combine to exhibit Hilbert--Schmidt stable groups of arbitrarily fast stability radius growth (Theorem \ref{intro:thm:arbitrary large SRad}).  

Recall the generalized B.H. Neumann groups  introduced in  Definition \ref{def:generalised B H Neumann}. Specifically, given a  strictly increasing sequence of odd integers $d(n)$ with $d(1)  \geq 5$, and a function $r: \N \to \N$ with $2r(n) \leq d(n)-1$ for all $n$, the corresponding $(d,r)$-generalized B.H. Neumann group $B(d,r)$ is a $2$-generated subgroup of the direct product $\prod_n \Alt(d(n))$ with a particular generating set $S = \{\alpha,\beta\}$.

For an   appropriate choice of the two functions $d$ and $r$, the MAP growth of the $(d,r)$-generalized B.H. Neumann group  can be arbitrarily large.
The  proof of this statement given below is a straightforward adaptation of the proof of the analogous statement on the residual finiteness growth of these groups, which was first established in \cite{BS}.
%Combined with the relationship between LEF, MAP and stability radius growths (Proposition \ref{prop:SRad_LEF_bound}), Theorem \ref{intro:thm:arbitrary large SRad} will follow.

\begin{proposition}
   \label{prop:fast MAP growth}
Let $f: \N \to \N$ be any monotone non-decreasing function. There are two functions $d,r : \N \to \N$ as above, such that the corresponding $(d,r)$-generalized B.H. Neumann group $B = B(d,r)$ satisfies $f \preceq \MAP_{B}$.
\end{proposition}

\begin{proof}
Let an arbitrary monotone non-decreasing function $f : \N \to \N$ be given. By making use of \cite[Proposition 2.8]{Bradford_RF_growths} together with  Bernard's postulate, one can find a strictly increasing sequence $d(n)$ of (odd) primes with $d(1) \ge 5$, and a function $r : \N \to \N$, with the following properties:
    \begin{enumerate}[label=(\alph*)]
        \item The two inequalities $n \leq r(n) \leq \min \{18n,  d(n)/3\}$ and $f(n) \leq d(n) -1$ hold true for all $n$.
        \item $r(l)$ is distinct modulo $d(m)$ from  $\pm r(m)$ as well as from $\pm 2r(m)$, for all $m$ and for all $l$ distinct from $m$.
    \end{enumerate}
    
Consider the corresponding generalized B.N. Neumann group $B = B(d,r)$ with its  generators $S = \{\alpha,\beta\}$. By (a) above,  the assumptions of Proposition \ref{prop:Bradford groups are diag prod} are satisfied. Therefore $B$ is isomorphic to the diagonal product $\bigotimes G_n$ where $G_n = \Alt(d(n))$ for all $n$. Note that  $\bigoplus_n G_n \leq B$ by  Lemma \ref{lem:sum_contained_in_diag_prod}.
Both items (a) and (b) above are used in \cite[Lemma 3.1 and Lemma 3.2]{Bradford_RF_growths} to show that for every $n$, there exists a non-trivial word $w_n \in B$ with $$w_n \in B_S(4 + 4r(n)) \cap G_n$$ where $G_n$ is viewed as a subgroup of the direct sum $\bigoplus G_n$ inside $B$.
Consequently, if $\pi: B \to U(\mathcal H)$  is any finite-dimensional representation such that $\pi(w_n) \neq 1$, then  the restriction of the representation $\pi$ to the subgroup $G_n \leq B$ is a non-trivial unitary representation of $G_n$.
    It is well known that the minimal non-trivial representation of $G_n \cong \Alt(d(n))$ is of dimension $d(n) - 1$ (achieved by the standard representation). Therefore $\dim \mathcal H \geq d(n) - 1$. We conclude that
    $$ \MAP_B^S(18n) \geq \MAP_B^S(4 + 4r(n)) \geq d(n) - 1 \geq f(n). $$
This implies $f \preceq \MAP_{B}$, as required.
\end{proof}

Another crucial ingredient towards Theorem \ref{intro:thm:arbitrary large SRad} is a universal upper bound on the LEF growth of all $2$-generated (or, more generally, all $k$-generated) groups.

\begin{theorem}[{\cite[Theorem 1.4]{Bradford_LEF}}]\label{thm:Bradford LEF bound}
Fix $k \in \N$. There exists an increasing function $\Xi_k: \N \to \N$ such that  any $k$-generated LEF group $\Gamma$ satisfies $\LEF_\Gamma \preceq \Xi_k$.
\end{theorem}

We can now conclude our construction of amenable groups with arbitrary large stability radius growth function.

\begin{proof}[Proof of Theorem \ref{intro:thm:arbitrary large SRad}]
Let $f: \N \to \N$ be a given monotone non-decreasing function. Assume  without loss of generality that $n \preceq f(n)$. %Furthermore $f \preceq \RF_{B_f} $, see the statement of Theorem \ref{thm:BS_grps}. 
Recall that there is a universal function $\Xi_2$ such that every $2$-generated group $\Gamma$ has $\LEF_\Gamma \preceq \Xi_2$, see Theorem \ref{thm:Bradford LEF bound}.
Define $g = \Xi_2 \circ f$. Let $B = B(d,r)$ be some generalised B. H. Neumann group associated with the two functions $d(n), r(n)$ obtained by applying Proposition \ref{prop:fast MAP growth} to the function $g$.

The group $B$ is $2$-generated and amenable. It is a diagonal product based on the wreath product $(\Z / 3\Z) \wr \Z$, see Proposition \ref{prop:Bradford groups are diag prod}. Recall that such wreath products are Hilbert--Schmidt stable by \cite[Corollary D]{levit2023characters}. Since the diagonal product construction preserves Hilbert--Schmidt stability by our Theorem \ref{thm:alt diag prod is stable}, we conclude that the group $B$ is also   Hilbert--Schmidt stable.

Putting together the two Propositions \ref{prop:SRad_LEF_bound} and \ref{prop:fast MAP growth} we obtain
$$ g(n) \preceq \MAP_{B}(n) \preceq \LEF_{B}(\max\{\SR_{B}(n), n\}) \preceq \Xi_2(\max\{ \SR_{B}(n), n\}) $$
for all $n$.
This gives  $f(n) \preceq \max\{ \SR_{B}(n), n\}$. Since $n \preceq f(n)$, we get $f \preceq \SR_{B}$, which concludes the proof.
\end{proof}

\begin{remark}
    There is a doubly exponential upper bound on the LEF growth of the groups $B(d,r)$ \cite[Theorem 4.9]{Bradford_LEF}. So, one can in fact use $\exp \circ \exp$ instead of $\Xi_2$.
\end{remark}

\begin{question}
Can one find upper bounds on the stability radius growth of some class of infinitely presented groups? For instance, is there an explicit (e.g.  computable) upper bound on the stability radius growth of the lamplighter group?
\end{question}

\footnotesize

\bibliographystyle{alpha}
\bibliography{ref}

\begin{thebibliography}{BdlHV14}

\bibitem[AP15]{AP}
G.~Arzhantseva and L.~P{\u{a}}unescu.
\newblock Almost commuting permutations are near commuting permutations.
\newblock {\em Journal of Functional Analysis}, 269(3):745--757, 2015.

\bibitem[BdlH20]{BdlH}
B.~Bekka and P.~de~la Harpe.
\newblock {\em Unitary representations of groups, duals, and characters},
  volume 250 of {\em Mathematical Surveys and Monographs}.
\newblock American Mathematical Society, Providence, RI, 2020.

\bibitem[BdlHV14]{bekka2014kazhdan}
B.~Bekka, P.~de~la Harpe, and A.~Valette.
\newblock {\em Kazhdan's Property {(T)}}.
\newblock New Mathematical Monographs. Cambridge University Press, 2014.

\bibitem[Bek07]{Bekka}
B.~Bekka.
\newblock Operator-algebraic superridigity for {${\rm SL}_n(\Bbb Z)$}, {$n\geq
  3$}.
\newblock {\em Invent. Math.}, 169(2):401--425, 2007.

\bibitem[BF22]{BF}
B.~Bekka and C.~Francini.
\newblock Characters of algebraic groups over number fields.
\newblock {\em Groups Geom. Dyn.}, 16(4):1119--1164, 2022.

\bibitem[BLT19]{BLT}
O.~Becker, A.~Lubotzky, and A.~Thom.
\newblock Stability and invariant random subgroups.
\newblock {\em Duke Mathematical Journal}, 168(12):2207--2234, 2019.

\bibitem[BM21]{becker2021abelian}
O.~Becker and J.~Mosheiff.
\newblock Abelian groups are polynomially stable.
\newblock {\em International Mathematics Research Notices},
  2021(20):15574--15632, 2021.

\bibitem[Bra22a]{Bradford_LEF_ext}
H.~Bradford.
\newblock Controlling {LEF} growth in some group extensions.
\newblock {\em arXiv preprint arXiv:2201.05052}, 2022.

\bibitem[Bra22b]{Bradford_LEF}
H.~Bradford.
\newblock Quantifying local embeddings into finite groups.
\newblock {\em Journal of Algebra}, 608:214--238, 2022.

\bibitem[Bra24a]{Bradford_RF_growths}
H.~Bradford.
\newblock On the spectrum of residual finiteness growth functions.
\newblock {\em arXiv preprint arXiv:2402.03556}, 2024.

\bibitem[Bra24b]{Bradford_stab}
Henry Bradford.
\newblock Local permutation stability.
\newblock {\em Groups, Geometry, and Dynamics}, 2024.

\bibitem[BRS16]{BS}
K.~Bou-Rabee and B.~Seward.
\newblock Arbitrarily large residual finiteness growth.
\newblock {\em J. Reine Angew. Math.}, 710:199--204, 2016.

\bibitem[BZ21]{BZ}
J.~Brieussel and T.~Zheng.
\newblock Speed of random walks, isoperimetry and compression of finitely
  generated groups.
\newblock {\em Ann. Math. (2)}, 193(1):1--105, 2021.

\bibitem[CM84]{carey1984characters}
A.L. Carey and W.~Moran.
\newblock Characters of nilpotent groups.
\newblock In {\em Mathematical Proceedings of the Cambridge Philosophical
  Society}, volume~96, pages 123--137. Cambridge University Press, 1984.

\bibitem[CVY23]{CVY}
M.~Chapman, T.~Vidick, and H.~Yuen.
\newblock Efficiently stable presentations from error-correcting codes.
\newblock {\em arXiv preprint arXiv:2311.04681}, 2023.

\bibitem[dlS22]{delaSalle}
M.~de~la Salle.
\newblock Spectral gap and stability for groups and non-local games.
\newblock {\em arXiv preprint arXiv:2204.07084}, 2022.

\bibitem[DM96]{dixon1996permutation}
J.D. Dixon and B.~Mortimer.
\newblock {\em Permutation groups}, volume 163.
\newblock Springer Science \& Business Media, 1996.

\bibitem[Dog21]{DogonAlon2021SaAo}
A.~Dogon.
\newblock {\em Stability and approximation of groups and operator algebras}.
\newblock M.Sc. thesis. Published by the Hebrew University, 2021.

\bibitem[Dog23]{Dog}
A.~Dogon.
\newblock Flexible {Hilbert}-{Schmidt} stability versus hyperlinearity for
  property ({T}) groups.
\newblock {\em Math. Z.}, 305(4):20, 2023.
\newblock Id/No 58.

\bibitem[ES23]{ES}
C.~Eckhardt and T.~Shulman.
\newblock On amenable {H}ilbert-{S}chmidt stable groups.
\newblock {\em J. Funct. Anal.}, 285(3):Paper No. 109954, 2023.

\bibitem[FFGS25]{FFGS}
Francesco Fournier-Facio, Maria Gerasimova, and Pieter Spaas.
\newblock Local hilbert–schmidt stability.
\newblock {\em Journal of Algebra}, 663:589--629, 2025.

\bibitem[Gle10]{Glebsky}
L.~Glebsky.
\newblock Almost commuting matrices with respect to normalized hilbert-schmidt
  norm.
\newblock {\em arXiv preprint arXiv:1002.3082}, 2010.

\bibitem[Glu97]{gluck1997characters}
D.~Gluck.
\newblock Characters and random walks on finite classical groups.
\newblock {\em Advances in Mathematics}, 1(129):46--72, 1997.

\bibitem[GR09]{GR}
L.~Glebsky and L.~M. Rivera.
\newblock Almost solutions of equations in permutations.
\newblock {\em Taiwanese Journal of Mathematics}, 13(2A):493--500, 2009.

\bibitem[GT24]{GT}
A.~Genevois and R.~Tessera.
\newblock Lamplighter-like geometry of groups.
\newblock {\em arXiv preprint arXiv:2401.13520}, 2024.

\bibitem[How77]{howe1977representations}
R.~Howe.
\newblock On representations of discrete, finitely generated, torsion-free,
  nilpotent groups.
\newblock {\em Pacific Journal of Mathematics}, 73(2):281--305, 1977.

\bibitem[HS18]{HS_grp}
D.~Hadwin and T.~Shulman.
\newblock Stability of group relations under small {Hilbert--Schmidt}
  perturbations.
\newblock {\em Journal of Functional Analysis}, 275(4):761--792, 2018.

\bibitem[Ioa24]{Ioana_HS}
Adrian Ioana.
\newblock Almost commuting matrices and stability for product groups.
\newblock {\em Journal of the European Mathematical Society}, 2024.

\bibitem[ISW20]{ISW}
A.~Ioana, P.~Spaas, and M.~Wiersma.
\newblock Cohomological obstructions to lifting properties for full {$\rm
  C^*$}-algebras of property ({T}) groups.
\newblock {\em Geom. Funct. Anal.}, 30(5):1402--1438, 2020.

\bibitem[JK81]{JK}
G.~James and A.~Kerber.
\newblock {\em The representation theory of the symmetric group}, volume~16 of
  {\em Encyclopedia of Mathematics and its Applications}.
\newblock Addison-Wesley Publishing Co., Reading, Mass., 1981.
\newblock With a foreword by P. M. Cohn, With an introduction by Gilbert de B.
  Robinson.

\bibitem[KP13]{KP}
M.~Kassabov and I.~Pak.
\newblock Groups of oscillating intermediate growth.
\newblock {\em Ann. Math. (2)}, 177(3):1113--1145, 2013.

\bibitem[LL22a]{LL1}
A.~Levit and A.~Lubotzky.
\newblock Infinitely presented permutation stable groups and invariant random
  subgroups of metabelian groups.
\newblock {\em Ergodic Theory Dynam. Systems}, 42(6):2028--2063, 2022.

\bibitem[LL22b]{LL2}
A.~Levit and A.~Lubotzky.
\newblock Uncountably many permutation stable groups.
\newblock {\em Israel Journal of Mathematics}, 251(2):657--678, 2022.

\bibitem[LL23]{lavi2023characters}
O.~Lavi and A.~Levit.
\newblock Characters of the group {$\mathrm{EL}_d (\mathcal{R})$} for a
  commutative {N}oetherian ring {$\mathcal{R}$}.
\newblock {\em Advances in Mathematics}, 419:108948, 2023.

\bibitem[LPS96]{LPS}
A.~Lubotzky, L.~Pyber, and A.~Shalev.
\newblock Discrete groups of slow subgroup growth.
\newblock {\em Isr. J. Math.}, 96:399--418, 1996.

\bibitem[LS08]{LS}
M.~Larsen and A.~Shalev.
\newblock Characters of symmetric groups: sharp bounds and applications.
\newblock {\em Invent. Math.}, 174(3):645--687, 2008.

\bibitem[LST11]{larsen2011waring}
M.~Larsen, A.~Shalev, and P.H. Tiep.
\newblock The {W}aring problem for finite simple groups.
\newblock {\em Annals of mathematics}, pages 1885--1950, 2011.

\bibitem[LV23]{levit2023characters}
A.~Levit and I.~Vigdorovich.
\newblock Characters of solvable groups, {H}ilbert--{S}chmidt stability and
  dense periodic measures.
\newblock {\em Mathematische Annalen}, pages 1--49, 2023.

\bibitem[LW93]{LW}
A.~Lubotzky and B.~Weiss.
\newblock Groups and expanders.
\newblock In {\em Expanding graphs ({P}rinceton, {NJ}, 1992)}, volume~10 of
  {\em DIMACS Ser. Discrete Math. Theoret. Comput. Sci.}, pages 95--109. Amer.
  Math. Soc., Providence, RI, 1993.

\bibitem[Mim19]{Mimura}
M.~Mimura.
\newblock Amenability versus non-exactness of dense subgroups of a compact
  group.
\newblock {\em J. Lond. Math. Soc., II. Ser.}, 100(2):592--622, 2019.

\bibitem[Neu37]{Neu}
B.~H. Neumann.
\newblock Some remarks on infinite groups.
\newblock {\em J. Lond. Math. Soc.}, 12:120--127, 1937.

\bibitem[Phe01]{phelps2001lectures}
R.R. Phelps.
\newblock {\em Lectures on Choquet’s theorem}.
\newblock Springer, 2001.

\bibitem[Roi96]{Roichman}
Y.~Roichman.
\newblock Upper bound on the characters of the symmetric groups.
\newblock {\em Invent. Math.}, 125(3):451--485, 1996.

\bibitem[R{\'S}08]{rattan2008upper}
A.~Rattan and P.~{\'S}niady.
\newblock Upper bound on the characters of the symmetric groups for balanced
  {Y}oung diagrams and a generalized {F}robenius formula.
\newblock {\em Advances in Mathematics}, 218(3):673--695, 2008.

\bibitem[Sku76]{skudlarek1976unzerlegbaren}
H.L. Skudlarek.
\newblock Die unzerlegbaren charaktere einiger diskreter gruppen.
\newblock {\em Mathematische Annalen}, 223(3):213--231, 1976.

\bibitem[Tho64a]{Thoma-symmetric}
E.~Thoma.
\newblock Die unzerlegbaren, positiv-definiten {K}lassenfunktionen der
  abz\"{a}hlbar unendlichen, symmetrischen {G}ruppe.
\newblock {\em Math. Z.}, 85:40--61, 1964.

\bibitem[Tho64b]{Thoma-characters}
E.~Thoma.
\newblock {\"U}ber unit{\"a}re darstellungen abz{\"a}hlbarer, diskreter
  gruppen.
\newblock {\em Mathematische Annalen}, 153(2):111--138, 1964.

\bibitem[Tho18]{ThomICM}
A.~Thom.
\newblock Finitary approximations of groups and their applications.
\newblock In {\em Proceedings of the {I}nternational {C}ongress of
  {M}athematicians---{R}io de {J}aneiro 2018. {V}ol. {III}. {I}nvited
  lectures}, pages 1779--1799. World Sci. Publ., Hackensack, NJ, 2018.

\bibitem[Tho22]{thomas2022characters}
S.~Thomas.
\newblock Characters and invariant random subgroups of the finitary symmetric
  group.
\newblock {\em Advances in Mathematics}, 399:108272, 2022.

\bibitem[Ula60]{Ulam}
S.~M. Ulam.
\newblock {\em A collection of mathematical problems}.
\newblock Interscience Tracts in Pure and Applied Mathematics, no. 8.
  Interscience Publishers, New York-London, 1960.

\bibitem[VG97]{VG}
A.~M. Vershik and E.~I. Gordon.
\newblock Groups that are locally embeddable in the class of finite groups.
\newblock {\em Algebra i Analiz}, 9(1):71--97, 1997.

\bibitem[VK81]{vershik1981asymptotic}
A.~M. Vershik and S.~V. Kerov.
\newblock Asymptotic theory of characters of the symmetric group.
\newblock {\em Functional analysis and its applications}, 15(4):246--255, 1981.

\end{thebibliography}

\vspace{0.5cm}

\noindent{\textsc{Department of Mathematics, Weizmann Institute of Science, Israel}}

\noindent{\textit{Email address:} \texttt{alon.dogon@weizmann.ac.il}} \\

\noindent{\textsc{Pure Mathematics Department, Tel Aviv University, Israel}}

\noindent{\textit{Email address:} \texttt{arielevit@tauex.tau.ac.il}} \\

\noindent{\textsc{Department of Mathematics, University of California San Diego, 9500 Gilman Drive, La Jolla, CA
92093, USA}}

\noindent{\textit{Email address:} \texttt{ividorovich@ucsd.edu}} \\
\end{document}